\documentclass{amsart}%
\usepackage{amsmath}
\usepackage{amsfonts}
\usepackage{amssymb}
\usepackage{graphicx}
\usepackage{chngcntr}%
\providecommand{\U}[1]{\protect\rule{.1in}{.1in}}
\counterwithin{equation}{section}
\newtheorem{theorem}{Theorem}[section]

\newtheorem{corollary}[theorem]{Corollary}

\newtheorem{lemma}[theorem]{Lemma}

\newtheorem{proposition}[theorem]{Proposition}
\newtheorem{remark}[theorem]{Remark}

\begin{document}
\title[Graham Theorem on Bounded Symmetric Domains]{Graham Theorem on Bounded Symmetric Domains}

\begin{abstract}
Graham Theorem on the unit ball $B_{n}$ in $\mathbb{C}^{n}$ states that every
invariant harmonic function $u\in C^{n}(\overline{B}_{n})$ must be
pluriharmonic in $B_{n}$ \cite{graham19831}. This rigidity phenomenon of
Graham have been studied by many authors ( see, for examples, \cite{graham88},
\cite{li02}, \cite{li10}, etc.) In this paper, we prove that Graham theorem
holds on classical bounded symmetric domains. Which include Type I domains,
Type II domains, Type III domains III(n) with even $n$ and some special Type
IV domains.

\end{abstract}
\keywords{Invariant Harmonic, Bounded Symmetric Domains, Pluriharmonic, Laplace-Beltrami Operator}
\author{Ren-Yu Chen and Song-Ying Li}
\address{\noindent Ren-Yu Chen: School of Mathematics, Tianjin University, Tianjin,
300354, People's Republic of China. \ E-mail: chenry@tju.edu.cn \hspace{4cm}
$\ \ \ \ \ \ \ \ \ \ \ \ \ \ \ \ \ \ \ \ \ \ \ \ \ \ \ \ \ \ \ \ $
$\ \ \ \ \ $ \hspace{0.2cm} ${ }$ \quad Song-Ying Li: Department of
Mathematics, University of California, Irvine, CA 92697-3875, USA. \ E-mail: sli@math.uci.edu}
\thanks{The first author was supported in part by the National Science Foundation of
China (Grant Nos: 11401426, 11371276).}
\thanks{The second author was supported in part by Ming-Jiang Scholar Fund when he was
visiting Fujian Normal University }
\subjclass[2010]{Primary 32A50; Secondary 32T15, 35J70, 35R01}
\maketitle

\section{Introduction}

\noindent Let $\left(  M^{n},g\right)  $ be a compact Riemannian with boundary
$\partial M$ with Riemannian metric $g$. Let $\Delta_{g}$ be the
Laplace-Beltrami operator associated to $g$. We consider the boundary value
problem%
\begin{equation}
\left\{
\begin{array}
[c]{cc}%
\Delta_{g}u=0, & \text{in }M,\\
u=\phi, & \text{on }\partial M.
\end{array}
\right.  \label{e1.1}%
\end{equation}
When $\Delta_{g}$ is uniformly elliptic on $M,$ the boundary value problem
$\left(  \ref{e1.1}\right)  $ is well understood (see, for examples, the books
of Evans \cite{evans10} and Gilbarg and Trudinger \cite{gilbarg83}). When
$\Delta_{g}$ is not uniformly elliptic, the regularity of the solution $u$ of
$\left(  \ref{e1.1}\right)  $ becomes much more complicated. Typical examples
we consider here are manifolds $\left(  M^{n},g\right)  $ with bounded
pseudoconvex domains $M$ in $\mathbb{C}^{n}$ and the Bergman metric $g$ of
$M$. In particular, when $M$ is the unit ball $B_{n}$ in $\mathbb{C}^{n},$ it
is well known from the books of Hua \cite{Hua19630} that%
\begin{equation}
\Delta_{g}=\left(  1-\left\vert z\right\vert ^{2}\right)  \sum_{\alpha
,\beta=1}^{n}\left(  \delta_{\alpha\beta}-z_{\alpha}\bar{z}_{\beta}\right)
\frac{\partial^{2}}{\partial z_{\alpha}\partial\bar{z}_{\beta}}. \label{e1.2}%
\end{equation}
For any $\phi\in C\left(  \partial B_{n}\right)  $, the Dirichlet boundary
value problem $\left(  \ref{e1.1}\right)  $ has a unique solution%
\begin{equation}
u\left(  z\right)  =\int_{\partial B_{n}}P\left(  z,w\right)  \phi\left(
w\right)  d\sigma\left(  w\right)  =\int_{\partial B_{n}}\frac{\left(
1-\left\vert z\right\vert ^{2}\right)  ^{n}}{\left\vert 1-\left\langle
z,w\right\rangle \right\vert ^{2n}}\phi\left(  w\right)  d\sigma\left(
w\right)  . \label{e1.3}%
\end{equation}
When $\phi\in C^{\infty}\left(  \partial B_{n}\right)  $, it was proved by
Graham \cite{graham19831} that the solution $u$ given by $\left(
\ref{e1.3}\right)  $ can be expressed as%
\begin{equation}
u\left(  z\right)  =G\left(  z\right)  +H\left(  z\right)  \left(
1-\left\vert z\right\vert ^{2}\right)  ^{n}\log\left(  1-\left\vert
z\right\vert ^{2}\right)  ,\;\;z\in\mathbb{B}_{n}, \label{e1.4}%
\end{equation}
where $G,H\in C^{\infty}\left(  \overline{B_{n}}\right)  $. When $n=1,$ $H\in
C_{0}^{\infty}\left(  B_{n}\right)  $. However, when $n>1,$ in general,
$H\not \equiv 0$ on $\partial B_{n}$. In particular, when $H=0$ on $\partial
B_{n},$ the following striking theorem was proved by Graham \cite{graham19831}:

\begin{theorem}
\label{t1.1}If $u\in C^{n}\left(  \overline{B_{n}}\right)  $ is invariant
harmonic $\left(  \Delta_{g}u=0\right)  $ in $B_{n},$ then $u$ is
pluriharmonic in $B_{n}.$
\end{theorem}

Problem about whether Graham's Phenomenon holds for more general domains $M$
and more general metric $g$ has been studied by several authors. For examples,
Graham and Lee \cite{graham88} studied the problem for $M$ being strictly
pseudoconvex domains in $\mathbb{C}^{n}$ with smooth boundaries and K\"{a}hler
metrics $g$ satisfying special symmetric property. In particular, they gave a
characterization of CR-pluriharmonic functions on $\partial M,$ which is a
fundamental paper in the theory of the pseudo-Hermitian CR geometry. Li and
Simon \cite{li02} proved a Graham type theorem for the polydisc in
$\mathbb{C}^{n}$ with Bergman type metrics. In general, Graham's phenomenon
fails in $B_{n}$ with rotationally symmetric metrics when $n>2,$
counterexample was constructed by Graham and Lee \cite{graham88}. Further
information along this direction can be found in Li and Wei \cite{li10}. For
more results on invariant harmonic functions and backgrounds we refer the
reader to \cite{ahern96}, \cite{graham19832}, \cite{krantz92}, \cite{li00},
\cite{liu04}, \cite{Rudin80}. However, the problem about whether Graham
Theorem holds when $M$ is a classical bounded symmetric domain with Bergman
metric $g$ is widely open. The main purpose of the paper is to investigate
this problem.

For positive integers $m\leq n,$ denoted by $M^{m,n}\left(  \mathbb{C}\right)
$ the set of all $m\times n$ matrices with entries in $\mathbb{C}$. The
classical bounded symmetric domains \cite{Hua19630} are the following four
types:%
\begin{equation}
\mathbf{I}\left(  m,n\right)  =\left\{  \mathbf{z}_{{}}\in M^{m,n}\left(
\mathbb{C}\right)  :I_{m}-\mathbf{z}_{{}}\mathbf{z}_{{}}^{\ast}>0\right\}  ,
\label{e1.5}%
\end{equation}%
\begin{equation}
\mathbf{II}\left(  n\right)  =\left\{  \mathbf{z}_{{}}\in\mathbf{I}\left(
n,n\right)  :\mathbf{z}_{{}}^{t}=\mathbf{z}_{{}}\right\}  , \label{ee1.6}%
\end{equation}%
\begin{equation}
\mathbf{III}\left(  n\right)  =\left\{  \mathbf{z}_{{}}\in\mathbf{I}\left(
n,n\right)  :\mathbf{z}_{{}}^{t}=-\mathbf{z}_{{}}\right\}  \label{e1.6}%
\end{equation}
and%
\begin{equation}
\mathbf{IV}\left(  n\right)  =\left\{  z\in\mathbb{C}^{n}:2\left\vert
z\right\vert ^{2}-\left\vert z_{{}}z^{t}\right\vert ^{2}-1<0\text{ and
}\left\vert z_{{}}z^{t}\right\vert ^{2}<1\right\}  . \label{e1.7}%
\end{equation}
Denoted by $D_{1}\cong D_{2}$ if $D_{1}$ and $D_{2}$ are biholomorphic
equivalent. It is known from Lu \cite{lu1963} and Loos \cite{loos1977} that%
\begin{equation}
\mathbf{IV}\left(  2\right)  \cong\mathbf{IV}\left(  1\right)  \times
\mathbf{IV}\left(  1\right)  ,\text{\ }B_{3}=\mathbf{I}\left(  1,3\right)
\cong\mathbf{III}\left(  3\right)  , \label{e1.8}%
\end{equation}
and%
\begin{equation}
\mathbf{II}\left(  2\right)  \cong\mathbf{IV}\left(  3\right)  ,\;\mathbf{I}%
\left(  2,2\right)  \cong\mathbf{IV}\left(  4\right)  \text{\ and
}\mathbf{III}\left(  4\right)  \cong\mathbf{IV}\left(  6\right)  .
\label{e1.9}%
\end{equation}
Let $D$ be a bounded domain with the Bergman metric $g^{D}$ of $D$. Let
$\Delta_{g^{D}}$ denote the Laplace-Beltrami operator associated to $g^{D}.$
Since the Bergman metric is biholomorphic invariant, we say that a function
$u$ is invariant harmonic in $D$ if%
\begin{equation}
\Delta_{g^{D}}u\left(  \mathbf{z}_{{}}\right)  =0,\;\;\mathbf{z}_{{}}\in D.
\label{e1.10}%
\end{equation}
When $D$ is a bounded symmetric domain, $\Delta_{g^{D}}$ is called Hua
operator. We use the following notations for Hua operators:%
\begin{equation}
\Delta_{1}=\Delta_{g^{\mathbf{I}\left(  m,n\right)  }},\Delta_{2}%
=\Delta_{g^{\mathbf{II}\left(  n\right)  }},\Delta_{3}=\Delta_{g^{\mathbf{III}%
\left(  n\right)  }}\text{ and }\Delta_{4}=\Delta_{g^{\mathbf{IV}\left(
n\right)  }}. \label{e1.11}%
\end{equation}
Denoted by $\mathcal{U}\left(  D\right)  $ the \v{S}ilov boundary or the
characteristic boundary of $D$. For any $\phi\in C\left(  \mathcal{U}\left(
D\right)  \right)  $, it was proved by Hua \cite{Hua19630}, the boundary value
problem $\left(  \ref{e1.1}\right)  $ has the unique solution%
\begin{equation}
u\left(  \mathbf{z}_{{}}\right)  =\int_{\mathcal{U}\left(  D\right)  }%
P^{D}\left(  \mathbf{z},\mathbf{w}\right)  \phi\left(  \mathbf{w}\right)
d\sigma\left(  \mathbf{w}\right)  , \label{e1.12}%
\end{equation}
where $P^{D}\left(  \mathbf{z}_{{}},\mathbf{w}_{{}}\right)  $ is the
Poisson-Szeg\"{o} kernel given by%
\begin{equation}
P^{D}\left(  \mathbf{z}_{{}},\mathbf{w}_{{}}\right)  =\frac{\left(
\det\left(  I_{m}-\mathbf{z}_{{}}\mathbf{z}_{{}}^{\ast}\right)  \right)
^{\kappa\left(  D\right)  }}{\left\vert \det\left(  I_{m}-\mathbf{z}_{{}%
}\mathbf{\mathbf{w}_{{}}^{\ast}}\right)  \right\vert ^{2\kappa\left(
D\right)  }} \label{e1.13}%
\end{equation}
and%
\begin{equation}
\kappa\left(  D\right)  =\left\{
\begin{array}
[c]{ll}%
n, & \text{if }D=\mathbf{I}\left(  m,n\right)  ;\\
\frac{n+1}{2}, & \text{if }D=\mathbf{II}\left(  n\right)  ;\\
\frac{n-1}{2}, & \text{if }D=\mathbf{III}\left(  n\right)  ,\;n\text{ is
even;}\\
\frac{n}{2} & \text{if }D=\mathbf{III}\left(  n\right)  ,\text{\ }n\;\text{is
odd.}%
\end{array}
\right.  \label{e1.14}%
\end{equation}

The main results of the paper are the following theorem.

\begin{theorem}
\label{t1.2}Let $m,n\in\mathbb{N}$ and $m\leq n.$

\begin{enumerate}
\item[(i)] If $u\in C^{n}\left(  \overline{\mathbf{I}\left(  m,n\right)
}\right)  $ is invariant harmonic then $u$ is pluriharmonic;

\item[(ii)] If $n$ is odd and if $u\in C^{\frac{n+1}{2}}\left(  \overline
{\mathbf{II}\left(  n\right)  }\right)  $ is invariant harmonic then $u$ is pluriharmonic;

\item[(iii)] If $n=2k$ is even and if there is a $\alpha>1/2$ such that $u\in
C^{k,\alpha}\left(  \overline{\mathbf{II}\left(  n\right)  }\right)  $ is
invariant harmonic then $u$ is pluriharmonic;

\item[(iv)] If $n$ is even and if $u\in C^{n-1}\left(  \overline
{\mathbf{III}\left(  n\right)  }\right)  $ is invariant harmonic then $u$ is pluriharmonic.
\end{enumerate}
\end{theorem}

\noindent\textbf{Remark} We note that the above smoothness assumptions are
sharp. The condition $C^{k,\alpha}$ with $\alpha>1/2$ in Part (iii) is the
same as ${\frac{n+1}{2}}+\epsilon$ with $\epsilon>0$. We need to add
$\epsilon$ in Part (iii) rather that ${\frac{n+1}{2}}$ in Part (ii) because
${\frac{n+1}{2}}$ is not integer when $n$ is even.

\vspace{0.3cm}

In order to prove Parts (ii) and (iii) of Theorem \ref{t1.2}, one of our key
steps is to prove the following theorem in the unit ball.

\begin{theorem}
\label{t1.3}Let%
\begin{equation}
\tilde{\Delta}:=\sum_{j,k=1}^{n}\left(  \delta_{jk}-\left\vert z\right\vert
^{2}z_{j}\bar{z}_{k}\right)  \frac{\partial^{2}}{\partial z_{j}\partial\bar
{z}_{k}}. \label{e1.22}%
\end{equation}
Then the following two statements hold.

\begin{enumerate}
\item[(i)] If $n$ is odd and $u\in C^{\frac{n+1}{2}}\left(  \overline{B_{n}%
}\right)  $ satisfying $\tilde{\Delta}u=0$ in $B_{n},$ then $u$ is
pluriharmonic in $B_{n};$

\item[(ii)] If $n=2k$ is even and if $u\in C^{k,\alpha}\left(  \partial
B_{n}\right)  $ for some $\alpha>1/2,$ then $u$ is pluriharmonic in $B_{n}.$
\end{enumerate}
\end{theorem}

\begin{remark}
We point out here that the operator $\tilde{\Delta}$ is not included in
$\Delta_{g}$ with $g$ is rotation symmetric metrics in Graham and Lee
\cite{graham88} or Li and Wei \cite{li10}.
\end{remark}

The paper is organized as follows. In section 2, we study the fundamental
properties of the Poisson-Szeg\"{o} kernels, we will prove that they satisfy a
system of differential equations. We will prove Theorem \ref{t1.3} in Section
3. As applications of results in Section 2, Grahams's theorem and Theorem
\ref{t1.3}, we will prove Theorem \ref{t1.2} in Section 4. Finally, in Section
5, we will prove Graham's phenomenon fails on $\mathbf{IV}\left(  2\right)  $
and will give some remarks on the problem over $\mathbf{III}\left(  3\right)
$ and $\mathbf{IV}\left(  4\right)  $.

\section{System of Differential Equations}

\noindent Let $\Delta_{1},\Delta_{2},\Delta_{3}$ and $\Delta_{4}$ denote Hua
operators, the Laplace-Beltrami operators associated to the Bergman metrics on
the classical bounded symmetric domains $\mathbf{I}\left(  m,n\right)  $,
$\mathbf{II}\left(  n\right)  $, $\mathbf{III}\left(  n\right)  $ and
\ $\mathbf{IV}\left(  n\right)  $, respectively. According to the books of Hua
\cite{Hua19630} and Lu \cite{lu1963}, the following proposition holds.

\begin{proposition}
\label{p2.1}For $\mathbf{z}\in M^{m,n}\left(  \mathbb{C}\right)  $ and
$V\left(  \mathbf{z}\right)  =I_{m}-\mathbf{z}\mathbf{z}^{\ast}$, we let%
\begin{equation}
V_{jk}=\left[  V\left(  \mathbf{z}\right)  \right]  _{jk}=\delta_{jk}%
-\sum_{\ell=1}^{n}z_{j\ell}\bar{z}_{k\ell}. \label{e2.1}%
\end{equation}
Then the Hua operators are given by%
\begin{equation}
\Delta_{1}=\sum_{j,k=1}^{m}V_{jk}\Delta_{1}^{jk},\;\;\Delta_{1}^{jk}%
:=\sum_{\alpha,\beta=1}^{n}\left(  \delta_{\alpha\beta}-\sum_{\ell=1}%
^{m}z_{\ell\alpha}\bar{z}_{\ell\beta}\right)  \frac{\partial^{2}}{\partial
z_{j\alpha}\partial\bar{z}_{k\beta}}; \label{e2.2}%
\end{equation}%
\begin{equation}
\Delta_{2}=\frac{1}{4}\sum_{j,k=1}^{n}V_{jk}\Delta_{2}^{jk},\;\;\Delta
_{2}^{jk}:=\sum_{\alpha,\beta=1}^{n}\frac{V_{\alpha\beta}}{\left(
1-\delta_{j\alpha}/2\right)  \left(  1-\delta_{k\beta}/2\right)  }%
\frac{\partial^{2}}{\partial z_{j\alpha}\partial\bar{z}_{k\beta}};
\label{e2.3}%
\end{equation}%
\begin{equation}
\Delta_{3}=\frac{1}{4}\sum_{j,k=1}^{n}V_{jk}\Delta_{3}^{jk},\;\;\Delta
_{3}^{jk}:=\sum_{\alpha,\beta=1}^{n}V_{\alpha\beta}\left(  1-\delta_{j\alpha
}\right)  \left(  1-\delta_{k\beta}\right)  \frac{\partial^{2}}{\partial
z_{j\alpha}\partial\bar{z}_{k\beta}} \label{e2.4}%
\end{equation}
and
\begin{equation}
\Delta_{4}=\sum_{j,k=1}^{n}\left[  r\left(  z\right)  \left(  \delta
_{jk}-2z_{j}\bar{z}_{k}\right)  +2\left(  \bar{z}_{j}-s\left(  \bar{z}\right)
z_{j}\right)  \left(  z_{k}-s\left(  z\right)  \bar{z}_{k}\right)  \right]
\frac{\partial^{2}}{\partial z_{j}\partial\bar{z}_{k}}, \label{e2.5}%
\end{equation}
where $z\in\mathbb{C}^{n}$ and
\begin{equation}
s\left(  z\right)  =\sum_{j=1}^{n}z_{j}^{2}\text{\ \ and}\;\;r\left(
z\right)  =1-2\left\vert z\right\vert ^{2}+2s\left(  z\right)  . \label{e2.6}%
\end{equation}

\end{proposition}

The main purpose of this section is to prove the following theorem.

\begin{theorem}
\label{t2.2}Let $D$ be a bounded symmetric domain and let $u\in C^{2}\left(
D\right)  \cap C\left(  \overline{D}\right)  $. Then

\begin{enumerate}
\item[(i)] $\Delta_{1}u=0$ on $\mathbf{I}\left(  m,n\right)  $ if and only if%
\[
\Delta_{1}^{jk}u\left(  \mathbf{z}\right)  =0,\;\text{ for all }1\leq j,k\leq
m.
\]

\item[(ii)] $\Delta_{2}u=0$ on $\mathbf{II}\left(  n\right)  $ if and only if%
\[
\Delta_{2}^{jk}u\left(  \mathbf{z}\right)  =0,\;\;1\leq j,k\leq n.
\]

\item[(iii)] When $n$ is even, $\Delta_{3}u\left(  \mathbf{z}\right)  =0$ on
$\mathbf{III}\left(  n\right)  $ if and only if%
\[
\Delta_{3}^{jk}u\left(  \mathbf{z}\right)  =0,\;\;1\leq j,k\leq n.
\]

\end{enumerate}
\end{theorem}

Part (i) of Theorem \ref{t2.2} was proved by Hua \cite{Hua19630} using the
technique of Lie group. We will divide the proof of the rest of the above
theorem into several lemmas. Define%
\begin{equation}
W\left(  \mathbf{z},\mathbf{w}\right)  =I_{m}-\mathbf{zw}^{\ast}%
\text{\ and}\;V\left(  \mathbf{z}\right)  =W\left(  \mathbf{z},\mathbf{z}%
\right)  . \label{e2.7}%
\end{equation}
By (\ref{e1.13}), the Poisson-Szeg\"{o} kernel on $D$ can be written as%
\begin{equation}
P^{D}\left(  \mathbf{z},\mathbf{w}\right)  =\frac{\left(  \det V\left(
\mathbf{z}\right)  \right)  ^{\kappa\left(  D\right)  }}{\left\vert \det
W\left(  \mathbf{z},\mathbf{w}\right)  \right\vert ^{2\kappa\left(  D\right)
}}. \label{e2.8}%
\end{equation}

\begin{proposition}
\label{p2.3}With the notations above, $P=P^{D}\left(  \mathbf{z}%
,\mathbf{w}\right)  $ and $\kappa=\kappa\left(  D\right)  $, one has%
\begin{equation}
\frac{1}{\kappa^{2}P}\frac{\partial^{2}P}{\partial z_{j\alpha}\partial\bar
{z}_{k\beta}}=\frac{1}{\kappa}\frac{\partial^{2}\log\det V\left(
\mathbf{z}\right)  }{\partial z_{j\alpha}\partial\bar{z}_{k\beta}}+\left(
b_{j\alpha}\overline{b_{k\beta}}+c_{j\alpha}\overline{c_{k\beta}}-b_{j\alpha
}\overline{c_{k\beta}}-c_{j\alpha}\overline{b_{k\beta}}\right)  , \label{e2.9}%
\end{equation}
where%
\begin{equation}
b_{j\alpha}:=\frac{\partial\log\det V\left(  \mathbf{z}\right)  }{\partial
z_{j\alpha}}\text{ and }c_{j\alpha}=c_{j\alpha}\left(  \mathbf{z}%
,\mathbf{w}\right)  :=\frac{\partial\log\det W\left(  \mathbf{z}%
,\mathbf{w}\right)  }{\partial z_{j\alpha}}. \label{e2.10}%
\end{equation}

\end{proposition}

\begin{proof}
Notice that $\det\overline{W\left(  \mathbf{z},\mathbf{w}\right)  }=\det
W\left(  \mathbf{w},\mathbf{z}\right)  $ and $\left(  \ref{e2.8}\right)  $,
one has%
\begin{equation}
\log P^{D}\left(  \mathbf{z},\mathbf{w}\right)  =\kappa\left(  \log\det
V\left(  \mathbf{z}\right)  -\log\det W\left(  \mathbf{z},\mathbf{w}\right)
-\log\det W\left(  \mathbf{\mathbf{w},\mathbf{z}}\right)  \right)
\label{e2.11}%
\end{equation}
and then%
\[
\kappa\frac{\partial^{2}\log\det V\left(  \mathbf{z}\right)  }{\partial
z_{j\alpha}\partial\bar{z}_{k\beta}}=\frac{\partial^{2}\log P^{D}\left(
\mathbf{z},\mathbf{w}\right)  }{\partial z_{j\alpha}\partial\bar{z}_{k\beta}%
}=\frac{1}{P}\frac{\partial^{2}P}{\partial z_{j\alpha}\partial\bar{z}_{k\beta
}}-\frac{\partial\log P}{\partial z_{j\alpha}}\frac{\partial\log P}%
{\partial\bar{z}_{k\beta}}.
\]
Therefore%
\[
\frac{1}{P}\frac{\partial^{2}P}{\partial z_{j\alpha}\partial\bar{z}_{k\beta}%
}=\kappa\frac{\partial^{2}\log\det V\left(  \mathbf{z}\right)  }{\partial
z_{j\alpha}\partial\bar{z}_{k\beta}}+\frac{\partial\log P}{\partial
z_{j\alpha}}\frac{\partial\log P}{\partial\bar{z}_{k\beta}}.
\]
A simple computation gives the proof of the proposition.
\end{proof}

Let $M\left(  \cdot\right)  =\left[  M_{jk}\left(  \cdot\right)  \right]  $ be
an $n\times n$ matrix and $M^{-1}\left(  \cdot\right)  =\left[  M^{jk}\left(
\cdot\right)  \right]  $. Here $j$ represents row index while $k$ represents
column index. Then%
\begin{equation}
\frac{\partial M\left(  \mathbf{z}\right)  }{\partial z_{j\alpha}}=\left[
\frac{\partial M_{pq}\left(  \mathbf{z}\right)  }{\partial z_{j\alpha}%
}\right]  . \label{e2.13}%
\end{equation}
Denoted by $\mathbf{E}_{jk}$ the $n\times n$ matrix with $\left(  j,k\right)
$-entry $1$, other entries $0.$

\begin{lemma}
\label{l2.4}

\begin{enumerate}
\item[(i)] For $\mathbf{z}\in\mathbf{II}\left(  n\right)  $ and $\mathbf{w}%
_{{}}\in\overline{\mathbf{II}\left(  n\right)  }$, one has%
\begin{equation}
c_{j\alpha}\left(  \mathbf{z},\mathbf{w}\right)  =\frac{\partial\log\det
W\left(  \mathbf{z},\mathbf{w}\right)  }{\partial z_{j\alpha}}=-\left(
2-\delta_{j\alpha}\right)  \left[  \mathbf{w}^{\ast}W^{-1}\left(
\mathbf{z},\mathbf{w}\right)  \right]  _{j\alpha} \label{e2.14}%
\end{equation}
and%
\begin{equation}
\overline{c_{k\beta}}\left(  \mathbf{z},\mathbf{w}\right)  =-\left(
2-\delta_{k\beta}\right)  \left[  W^{-1}\left(  \mathbf{w},\mathbf{z}\right)
\mathbf{w}\right]  _{k\beta}. \label{e2.15}%
\end{equation}

\item[(ii)] For $\mathbf{z}\in\mathbf{III}\left(  n\right)  $ and
$\mathbf{w}\in\mathbf{\overline{\mathbf{III}\left(  n\right)  }}$, one has%
\begin{equation}
c_{j\alpha}\left(  \mathbf{z},\mathbf{w}\right)  =\frac{\partial\log\det
W\left(  \mathbf{z},\mathbf{w}\right)  }{\partial z_{j\alpha}}=2\left[
\mathbf{w}^{\ast}W^{-1}\left(  \mathbf{z},\mathbf{w}\right)  \right]
_{j\alpha} \label{e2.16}%
\end{equation}
and%
\begin{equation}
\overline{c_{k\beta}}\left(  \mathbf{z},\mathbf{w}\right)  =-2\left[
W^{-1}\left(  \mathbf{w},\mathbf{z}\right)  \mathbf{w}\right]  _{k\beta}.
\label{e2.17}%
\end{equation}

\end{enumerate}
\end{lemma}

\begin{proof}
On $\mathbf{II}\left(  n\right)  $, $\mathbf{z}^{t}=\mathbf{z}$ and
$\mathbf{w}^{t}=\mathbf{w},$ one can easily check that%
\begin{equation}
\left(  \mathbf{w}^{\ast}W^{-1}\left(  \mathbf{z},\mathbf{w}_{{}}\right)
\right)  ^{t}=\mathbf{w}^{\ast}W^{-1}\left(  \mathbf{z},\mathbf{w}\right)
.\label{e2.18}%
\end{equation}
Thus
\begin{align*}
c_{j\alpha}\left(  \mathbf{z},\mathbf{w}\right)   &  =\operatorname*{tr}%
\left(  W^{-1}\left(  \mathbf{z},\mathbf{w}\right)  \frac{\partial W\left(
\mathbf{z},\mathbf{w}\right)  }{\partial z_{j\alpha}}\right)  \\
&  =-\operatorname*{tr}\left(  W^{-1}\left(  \mathbf{z},\mathbf{w}\right)
\left(  1-\frac{\delta_{j\alpha}}{2}\right)  \left(  \mathbf{E}_{j\alpha
}+E_{\alpha j}\right)  \mathbf{w}^{\ast}\right)  \\
&  =-\left(  1-\frac{\delta_{j\alpha}}{2}\right)  \operatorname*{tr}\left(
\mathbf{E}_{j\alpha}\mathbf{w}^{\ast}W^{-1}\left(  \mathbf{z},\mathbf{w}%
\right)  +\mathbf{w}^{\ast}W^{-1}\left(  \mathbf{z},\mathbf{w}\right)
\mathbf{E}_{\alpha j}\right)  \\
&  =-\left(  1-\frac{\delta_{j\alpha}}{2}\right)  \operatorname*{tr}\left(
\mathbf{E}_{j\alpha}\mathbf{w}^{\ast}W^{-1}\left(  \mathbf{z},\mathbf{w}%
\right)  +\mathbf{E}_{\alpha j}^{t}\left(  \mathbf{w}^{\ast}W^{-1}\left(
\mathbf{z},\mathbf{w}\right)  \right)  ^{t}\right)  \\
&  =-\left(  2-\delta_{j\alpha}\right)  \operatorname*{tr}\left(
\mathbf{E}_{j\alpha}\mathbf{w}^{\ast}W^{-1}\left(  \mathbf{z},\mathbf{w}%
\right)  \right)  \\
&  =-\left(  2-\delta_{j\alpha}\right)  \left[  \mathbf{w}^{\ast}W^{-1}\left(
\mathbf{z},\mathbf{w}\right)  \right]  _{j\alpha}.
\end{align*}
By $\left(  \ref{e2.18}\right)  $, one has $\left(  \ref{e2.15}\right)  $
holds and Part (i) is proved.

On $\mathbf{III}\left(  n\right)  $, it is easily to see that $\left(
\mathbf{w}^{\ast}W^{-1}\left(  \mathbf{z},\mathbf{w}\right)  \right)  ^{t}$
and $\left(  W^{-1}\left(  \mathbf{\mathbf{w}_{{}}},\mathbf{\mathbf{z}_{{}}%
}\right)  \mathbf{\mathbf{w}}\right)  ^{t}$ are anti-symmetric and%
\begin{align*}
c_{j\alpha}\left(  \mathbf{z},\mathbf{w}\right)   &  =\operatorname*{tr}%
\left(  W^{-1}\left(  \mathbf{z},\mathbf{w}\right)  \frac{\partial W\left(
\mathbf{z},\mathbf{w}\right)  }{\partial z_{j\alpha}}\right)  \\
&  =-\operatorname*{tr}\left(  W^{-1}\left(  \mathbf{z},\mathbf{w}\right)
\left(  \mathbf{E}_{j\alpha}-\mathbf{E}_{\alpha j}\right)  \mathbf{w}^{\ast
}\right)  \\
&  =-2\operatorname*{tr}\left(  \mathbf{E}_{j\alpha}\mathbf{w}^{\ast}%
W^{-1}\left(  \mathbf{z},\mathbf{w}\right)  \right)  \\
&  =-2\left[  \mathbf{w}^{\ast}W^{-1}\left(  \mathbf{z},\mathbf{w}\right)
\right]  _{\alpha j}\\
&  =2\left[  \mathbf{w}^{\ast}W^{-1}\left(  \mathbf{z},\mathbf{w}\right)
\right]  _{j\alpha}%
\end{align*}
and $\overline{c_{k\beta}}\left(  \mathbf{z},\mathbf{w}\right)  =-2\left[
W^{-1}\left(  \mathbf{w},\mathbf{z}\right)  \mathbf{w}\right]  _{k\beta}.$
Therefore, Part (ii) is proved and so is the lemma.
\end{proof}

\begin{lemma}
\label{l2.5}

\begin{enumerate}
\item[(i)] For $\mathbf{z}\in\mathbf{II}\left(  n\right)  $, one has%
\begin{equation}
A_{2}^{jk}\left(  \mathbf{z}_{{}}\right)  :=\frac{2}{n+1}\Delta_{2}^{jk}%
\log\det V=-4V^{kj} \label{e2.19}%
\end{equation}
and%
\begin{equation}
B_{2}^{jk}\left(  \mathbf{z}_{{}}\right)  :=\sum_{\alpha,\beta}\frac
{V_{\alpha\beta}}{\left(  1-\delta_{j\alpha}/2\right)  \left(  1-\delta
_{k\beta}/2\right)  }b_{j\alpha}\overline{b_{k\beta}}=4\left[  V^{-1}\left(
\mathbf{z}^{\ast}\right)  -I_{n}\right]  _{jk}. \label{e2.20}%
\end{equation}

\item[(ii)] For $\mathbf{z}_{{}}\in\mathbf{III}\left(  n\right)  $, one has%
\begin{equation}
A_{3}^{jk}\left(  \mathbf{z}_{{}}\right)  :=\frac{1}{\kappa}\Delta_{3}%
^{jk}\log\det V\left(  \mathbf{z}_{{}}\right)  =-\frac{2}{\kappa}\left(
n-1\right)  V^{kj}, \label{e2.21}%
\end{equation}
where $\kappa=\kappa\left(  \mathbf{III}\left(  n\right)  \right)  $ and%
\begin{equation}
B_{3}^{jk}\left(  \mathbf{z}_{{}}\right)  :=\sum_{\alpha,\beta}V_{\alpha\beta
}\left(  1-\delta_{j\alpha}\right)  \left(  1-\delta_{k\beta}\right)
b_{j\alpha}\overline{b_{k\beta}}=4\left[  V^{-1}\left(  \mathbf{z}_{{}}^{\ast
}\right)  -I_{n}\right]  _{jk}. \label{e2.22}%
\end{equation}

\end{enumerate}
\end{lemma}

\begin{proof}
On $\mathbf{II}\left(  n\right)  $, it is known from Lu \cite{lu1963}%
\[
\frac{\partial^{2}\log\det V\left(  \mathbf{z}_{{}}\right)  }{\partial
z_{j\alpha}\partial\bar{z}_{k\beta}}=-2\left(  1-\frac{\delta_{j\alpha}}%
{2}\right)  \left(  1-\frac{\delta_{\beta k}}{2}\right)  \left(  V^{\beta
j}V^{k\alpha}+V^{\beta\alpha}V^{kj}\right)  .
\]
This implies that%
\begin{align*}
A_{2}^{jk}\left(  \mathbf{z}_{{}}\right)   &  =\frac{2}{n+1}\Delta_{2}%
^{jk}\log\det V\left(  \mathbf{z}_{{}}\right) \\
&  =\frac{2}{n+1}\sum_{\alpha,\beta=1}^{n}\frac{V_{\alpha\beta}}{\left(
1-\delta_{j\alpha}/2\right)  \left(  1-\delta_{k\beta}/2\right)  }%
\frac{\partial^{2}\log\det V\left(  \mathbf{z}_{{}}\right)  }{\partial
z_{j\alpha}\partial\bar{z}_{k\beta}}\\
&  =-\frac{4}{n+1}\sum_{\alpha,\beta=1}^{n}V_{\alpha\beta}\left(  V^{\beta
j}V^{k\alpha}+V^{\beta\alpha}V^{kj}\right) \\
&  =-\frac{4}{n+1}\left(  V^{kj}+nV^{kj}\right) \\
&  =-4V^{kj}.
\end{align*}

With the notations $b_{j\alpha}\left(  \mathbf{z}_{{}}\right)  =c_{j\alpha
}\left(  \mathbf{z},\mathbf{z}\right)  $, $V\left(  \mathbf{z}\right)
=W\left(  \mathbf{z},\mathbf{z}\right)  $ and the fact that $V^{-1}\left(
\mathbf{z}\right)  \mathbf{z}$ is symmetric, Part (i) of Lemma \ref{l2.4}
implies%
\begin{align*}
B_{2}^{jk}\left(  \mathbf{z}\right)   &  =\sum_{\alpha,\beta}\frac
{V_{\alpha\beta}}{\left(  1-\delta_{j\alpha}/2\right)  \left(  1-\delta
_{k\beta}/2\right)  }c_{j\alpha}\left(  \mathbf{z}_{{}},\mathbf{z}_{{}%
}\right)  \overline{c_{k\beta}\left(  \mathbf{z}_{{}},\mathbf{z}_{{}}\right)
}\\
&  =4\sum_{\alpha,\beta}\frac{V_{\alpha\beta}}{\left(  1-\frac{\delta
_{j\alpha}}{2}\right)  \left(  1-\frac{\delta_{k\beta}}{2}\right)  }\left(
1-\frac{\delta_{j\alpha}}{2}\right)  \left[  \mathbf{z}_{{}}^{\ast}%
V^{-1}\left(  \mathbf{z}_{{}}\right)  \right]  _{j\alpha}\left(
1-\frac{\delta_{k\beta}}{2}\right)  \left[  V^{-1}\left(  \mathbf{z}_{{}%
}\right)  \mathbf{z}_{{}}\right]  _{k\beta}\\
&  =4\sum_{\alpha,\beta}V_{\alpha\beta}\left[  \mathbf{z}_{{}}^{\ast}%
V^{-1}\left(  \mathbf{z}_{{}}\right)  \right]  _{j\alpha}\left[  V^{-1}\left(
\mathbf{z}_{{}}\right)  \mathbf{z}_{{}}\right]  _{\beta k}\\
&  =4\left[  \mathbf{z}_{{}}^{\ast}V^{-1}\left(  \mathbf{z}_{{}}\right)
V\left(  \mathbf{z}\right)  V^{-1}\left(  \mathbf{z}_{{}}\right)
\mathbf{z}_{{}}\right]  _{jk}\\
&  =4\left[  \mathbf{z}_{{}}^{\ast}V^{-1}\left(  \mathbf{z}_{{}}\right)
\mathbf{z}_{{}}\right]  _{jk}\\
&  =4\left[  V^{-1}\left(  \mathbf{z}_{{}}^{\ast}\right)  -I_{n}\right]
_{jk}.
\end{align*}

On $\mathbf{III}\left(  n\right)  $, according to Lu \cite{lu1963}, one has%
\[
\frac{\partial^{2}\log\det V\left(  \mathbf{z}_{{}}\right)  }{\partial
z_{j\alpha}\partial\bar{z}_{k\beta}}=2\left(  V^{\beta j}V^{k\alpha}%
-V^{\beta\alpha}V^{kj}\right)  .
\]
This implies%
\begin{align*}
A_{3}^{jk}\left(  \mathbf{z}_{{}}\right)   &  =\frac{1}{\kappa}\Delta_{3}%
^{jk}\log\det V\left(  \mathbf{z}_{{}}\right)  \\
&  =\frac{1}{\kappa}\sum_{\alpha,\beta=1}^{n}\left(  1-\delta_{j\alpha
}\right)  \left(  1-\delta_{k\beta}\right)  V_{\alpha\beta}2\left(  V^{\beta
j}V^{k\alpha}-V^{\beta\alpha}V^{kj}\right)  \\
&  =\frac{2}{\kappa}\sum_{\alpha=1}^{n}\left(  1-\delta_{j\alpha}\right)
\left[  V^{k\alpha}\delta_{j\alpha}-V^{kj}-V_{\alpha k}V^{kj}V^{k\alpha
}+V_{\alpha k}V^{k\alpha}V^{kj}\right]  \\
&  =-\frac{2}{\kappa}\left(  n-1\right)  V^{kj}.
\end{align*}
Notice that $\mathbf{z}_{{}}^{\ast}V^{-1}\left(  \mathbf{z}_{{}}\right)  $ is
anti-symmetric, one has%
\[
\left(  1-\delta_{j\alpha}\right)  \left[  \mathbf{z}_{{}}^{\ast}V^{-1}\left(
\mathbf{z}_{{}}\right)  \right]  _{j\alpha}=\left[  \mathbf{z}_{{}}^{\ast
}V^{-1}\left(  \mathbf{z}_{{}}\right)  \right]  _{j\alpha}.
\]
Part (ii)$\ $of Lemma \ref{l2.4} implies%
\begin{align*}
B_{3}^{jk}\left(  \mathbf{z}\right)   &  =\sum_{\alpha,\beta}V_{\alpha\beta
}\left(  1-\delta_{j\alpha}\right)  \left(  1-\delta_{k\beta}\right)
c_{j\alpha}\left(  \mathbf{z}_{{}},\mathbf{z}_{{}}\right)  \overline
{c_{k\beta}\left(  \mathbf{z}_{{}},\mathbf{z}_{{}}\right)  }\\
&  =-4\sum_{\alpha,\beta}V_{\alpha\beta}\left(  1-\delta_{j\alpha}\right)
\left(  1-\delta_{k\beta}\right)  \left[  \mathbf{z}_{{}}^{\ast}V^{-1}\left(
\mathbf{z}_{{}}\right)  \right]  _{j\alpha}\left[  V^{-1}\left(
\mathbf{z}_{{}}\right)  \mathbf{z}_{{}}\right]  _{k\beta}\\
&  =4\left[  \mathbf{z}_{{}}^{\ast}V^{-1}\left(  \mathbf{z}_{{}}\right)
V\left(  \mathbf{z}\right)  V^{-1}\left(  \mathbf{z}_{{}}\right)
\mathbf{z}_{{}}\right]  _{jk}\\
&  =4\left[  \mathbf{z}^{\ast}V^{-1}\left(  \mathbf{z}_{{}}\right)
\mathbf{z}_{{}}\right]  _{jk}\\
&  =4\left[  V^{-1}\left(  \mathbf{z}_{{}}^{\ast}\right)  -I_{n}\right]
_{jk}.
\end{align*}
Therefore, the proof of the lemma is complete.
\end{proof}

\begin{lemma}
\label{l2.6}Let%
\begin{equation}
C_{2}^{jk}\left(  \mathbf{z}_{{}},\mathbf{w}_{{}}\right)  :=\sum_{\alpha
,\beta}\frac{V_{\alpha\beta}c_{j\alpha}\left(  \mathbf{z}_{{}},\mathbf{w}_{{}%
}\right)  \overline{c_{k\beta}}\left(  \mathbf{z}_{{}},\mathbf{w}_{{}}\right)
}{\left(  1-\delta_{j\alpha}/2\right)  \left(  1-\delta_{k\beta}/2\right)  }
\label{ee2.23}%
\end{equation}
and
\begin{equation}
C_{3}^{jk}\left(  \mathbf{z}_{{}},\mathbf{w}_{{}}\right)  :=\sum_{\alpha
,\beta}V_{\alpha\beta}\left(  1-\delta_{j\alpha}\right)  \left(
1-\delta_{k\beta}\right)  c_{j\alpha}\left(  \mathbf{z}_{{}},\mathbf{w}_{{}%
}\right)  \overline{c_{k\beta}}\left(  \mathbf{z}_{{}},\mathbf{w}_{{}}\right)
. \label{e2.23}%
\end{equation}
Then

\begin{enumerate}
\item[(i)] For $\mathbf{w}\in\mathcal{U}_{{}}\left(  \mathbf{II}\left(
n\right)  \right)  $ and $\mathbf{z}_{{}}\in\mathbf{II}\left(  n\right)  $,
one has%
\begin{equation}
C_{2}^{jk}\left(  \mathbf{z},\mathbf{w}\right)  =4\left[  W^{-1}\left(
\mathbf{z}_{{}}^{\ast},\mathbf{w}_{{}}^{\ast}\right)  +W^{-1}\left(
\mathbf{w}_{{}}^{\ast},\mathbf{z}_{{}}^{\ast}\right)  -I_{n}\right]  _{jk}.
\label{e2.24}%
\end{equation}

\item[(ii)] For $\mathbf{w}_{{}}\in\mathcal{U}_{{}}\left(  \mathbf{III}\left(
n\right)  \right)  $ and $\mathbf{z}_{{}}\in\mathbf{III}\left(  n\right)  $,
one has
\begin{equation}
C_{3}^{jk}\left(  \mathbf{z}_{{}},\mathbf{w}_{{}}\right)  =4\left[
W^{-1}\left(  \mathbf{z}_{{}}^{\ast},\mathbf{w}_{{}}^{\ast}\right)
+W^{-1}\left(  \mathbf{w}_{{}}^{\ast},\mathbf{z}_{{}}^{\ast}\right)
-I_{n}-F\left(  \mathbf{z}_{{}}\mathbf{,\mathbf{w}_{{}}}\right)  \right]
_{jk}, \label{e2.25}%
\end{equation}
where%
\begin{equation}
F\left(  \mathbf{z}_{{}},\mathbf{w}_{{}}\right)  :=W^{-1}\left(
\mathbf{w}_{{}}^{\ast},\mathbf{z}_{{}}^{\ast}\right)  \left(  I_{n}%
-\mathbf{w}_{{}}^{\ast}\mathbf{w}_{{}}\right)  W^{-1}\left(  \mathbf{z}_{{}%
}^{\ast},\mathbf{w}_{{}}^{\ast}\right)  . \label{e2.26}%
\end{equation}
In particular, when $n$ is even and $\mathbf{w}_{{}}\in\mathcal{U}_{{}}\left(
\mathbf{III}\left(  n\right)  \right)  $, one has $I_{n}-\mathbf{w}_{{}}%
^{\ast}\mathbf{w}_{{}}=0$ and $F\left(  \mathbf{z}_{{}},\mathbf{w}_{{}%
}\right)  =0.$\qquad
\end{enumerate}
\end{lemma}

\begin{proof}
On $\mathbf{II}\left(  n\right)  $, notice that $\mathbf{w}_{{}}\mathbf{w}%
_{{}}^{\ast}=I_{n},$%
\begin{equation}
\mathbf{w}_{{}}^{\ast}W^{-1}\left(  \mathbf{z}_{{}},\mathbf{w}_{{}}\right)
=W^{-1}\left(  \mathbf{z}_{{}},\mathbf{w}_{{}}\right)  ^{t}\mathbf{w}_{{}%
}^{\ast}\text{ and }W^{-1}\left(  \mathbf{w}_{{}},\mathbf{z}_{{}}\right)
\mathbf{w}_{{}}=\mathbf{w}_{{}}W^{-1}\left(  \mathbf{w}_{{}},\mathbf{z}_{{}%
}\right)  ^{t},\label{e2.27}%
\end{equation}
one has%
\begin{align*}
&  \mathbf{w}_{{}}^{\ast}W^{-1}\left(  \mathbf{z}_{{}},\mathbf{w}_{{}}\right)
V\left(  \mathbf{z}_{{}}\right)  W^{-1}\left(  \mathbf{w}_{{}},\mathbf{z}_{{}%
}\right)  \mathbf{w}_{{}}\\
&  =W^{-1}\left(  \mathbf{z}_{{}},\mathbf{w}_{{}}\right)  ^{t}\mathbf{w}_{{}%
}^{\ast}\left(  I_{n}-\mathbf{z}_{{}}\mathbf{z}_{{}}^{\ast}\right)
\mathbf{w}_{{}}W^{-1}\left(  \mathbf{w}_{{}},\mathbf{z}_{{}}\right)  ^{t}\\
&  =W^{-1}\left(  \mathbf{w}_{{}}^{\ast},\mathbf{z}_{{}}^{\ast}\right)
\left(  I_{n}-\mathbf{w}_{{}}^{\ast}\mathbf{z}_{{}}\mathbf{z}_{{}}^{\ast
}\mathbf{w}_{{}}\right)  W^{-1}\left(  \mathbf{z}_{{}}^{\ast},\mathbf{w}_{{}%
}^{\ast}\right)  \\
&  =W^{-1}\left(  \mathbf{w}_{{}}^{\ast},\mathbf{z}_{{}}^{\ast}\right)
\left(  W\left(  \mathbf{\mathbf{w}_{{}}^{\ast},\mathbf{z}_{{}}^{\ast}%
}\right)  +\left(  I_{n}-W\left(  \mathbf{w}_{{}}^{\ast},\mathbf{z}_{{}}%
^{\ast}\right)  \right)  W\left(  \mathbf{z}_{{}}^{\ast},\mathbf{w}_{{}}%
^{\ast}\right)  \right)  W^{-1}\left(  \mathbf{z}_{{}}^{\ast},\mathbf{w}_{{}%
}^{\ast}\right)  \\
&  =\left(  I_{n}+\left(  W^{-1}\left(  \mathbf{w}_{{}}^{\ast},\mathbf{z}_{{}%
}^{\ast}\right)  -I_{n}\right)  W\left(  \mathbf{z}_{{}}^{\ast},\mathbf{w}%
_{{}}^{\ast}\right)  \right)  W^{-1}\left(  \mathbf{z}_{{}}^{\ast}%
,\mathbf{w}_{{}}^{\ast}\right)  \\
&  =W^{-1}\left(  \mathbf{z}_{{}}^{\ast},\mathbf{w}_{{}}^{\ast}\right)
+W^{-1}\left(  \mathbf{w}_{{}}^{\ast},\mathbf{z}_{{}}^{\ast}\right)  -I_{n}.
\end{align*}
By Part (i) of Lemma \ref{l2.4} and the identity above, one has%
\begin{align*}
C_{2}^{jk}\left(  \mathbf{z}_{{}},\mathbf{w}_{{}}\right)   &  =\sum
_{\alpha,\beta=1}^{n}\frac{V_{\alpha\beta}}{\left(  1-\delta_{j\alpha
}/2\right)  \left(  1-\delta_{k\beta}/2\right)  }c_{j\alpha}\left(
\mathbf{z}_{{}},\mathbf{w}_{{}}\right)  \overline{c_{k\beta}\left(
\mathbf{z}_{{}},\mathbf{w}_{{}}\right)  }\\
&  =4\sum_{\alpha,\beta}V_{\alpha\beta}\left[  \mathbf{w}_{{}}^{\ast}%
W^{-1}\left(  \mathbf{z}_{{}},\mathbf{w}_{{}}\right)  \right]  _{j\alpha
}\left[  W^{-1}\left(  \mathbf{w}_{{}},\mathbf{z}_{{}}\right)  \mathbf{w}_{{}%
}\right]  _{k\beta}\\
&  =4\left[  \mathbf{w}_{{}}^{\ast}W^{-1}\left(  \mathbf{z}_{{}}%
,\mathbf{w}_{{}}\right)  V\left(  \mathbf{z}_{{}}\right)  W^{-1}\left(
\mathbf{w}_{{}},\mathbf{z}_{{}}\right)  \mathbf{w}_{{}}\right]  _{jk}\\
&  =4\left[  W^{-1}\left(  \mathbf{z}_{{}}^{\ast},\mathbf{w}_{{}}^{\ast
}\right)  +W^{-1}\left(  \mathbf{w}^{\ast},\mathbf{z}_{{}}^{\ast}\right)
-I_{n}\right]  _{jk}.
\end{align*}
Therefore, Part (i) is proved.\medskip

On $\mathbf{III}\left(  n\right)  $, $\mathbf{w}_{{}}^{\ast}W^{-1}\left(
\mathbf{z}_{{}},\mathbf{w}_{{}}\right)  $ and $W^{-1}\left(  \mathbf{w}_{{}%
},\mathbf{z}_{{}}\right)  \mathbf{w}_{{}}$ are anti-symmetric, by Part (ii) of
Lemma \ref{l2.4}, one has%
\begin{align*}
C_{3}^{jk}\left(  \mathbf{z}_{{}},\mathbf{w}_{{}}\right)   &  =\sum
_{\alpha,\beta=1}^{n}\left(  1-\delta_{j\alpha}\right)  \left(  1-\delta
_{k\beta}\right)  V_{\alpha\beta}c_{j\alpha}\left(  \mathbf{z}_{{}}%
,\mathbf{w}_{{}}\right)  \overline{c_{k\beta}\left(  \mathbf{z}_{{}%
},\mathbf{w}_{{}}\right)  }\\
&  =-4\sum_{\alpha,\beta=1}^{n}\left(  1-\delta_{j\alpha}\right)  \left(
1-\delta_{k\beta}\right)  V_{\alpha\beta}\left[  \mathbf{w}_{{}}^{\ast}%
W^{-1}\left(  \mathbf{z}_{{}},\mathbf{w}_{{}}\right)  \right]  _{j\alpha
}\left[  W^{-1}\left(  \mathbf{w}_{{}},\mathbf{z}_{{}}\right)  \mathbf{w}_{{}%
}\right]  _{k\beta}\\
&  =4\left[  \mathbf{w}^{\ast}W^{-1}\left(  \mathbf{z}_{{}},\mathbf{w}_{{}%
}\right)  V\left(  \mathbf{z}_{{}}\right)  W^{-1}\left(  \mathbf{w}_{{}%
},\mathbf{z}_{{}}\right)  \mathbf{w}_{{}}\right]  _{jk}\\
&  =4\left[  W^{-1}\left(  \mathbf{z}_{{}},\mathbf{w}_{{}}\right)
^{t}\mathbf{w}_{{}}^{\ast}V\left(  \mathbf{z}_{{}}\right)  \mathbf{w}_{{}%
}W^{-1}\left(  \mathbf{w}_{{}},\mathbf{z}_{{}}\right)  ^{t}\right]  _{jk}%
\end{align*}
and%
\begin{align*}
&  W^{-1}\left(  \mathbf{z}_{{}},\mathbf{w}_{{}}\right)  ^{t}\mathbf{w}_{{}%
}^{\ast}V\left(  \mathbf{z}_{{}}\right)  \mathbf{w}_{{}}W^{-1}\left(
\mathbf{w}_{{}},\mathbf{z}_{{}}\right)  ^{t}\\
&  =W^{-1}\left(  \mathbf{w}_{{}}^{\ast},\mathbf{z}_{{}}^{\ast}\right)
\left(  \mathbf{w}_{{}}^{\ast}\mathbf{w}_{{}}-\mathbf{w}_{{}}^{\ast}%
\mathbf{z}_{{}}\mathbf{z}_{{}}^{\ast}\mathbf{w}_{{}}\right)  W^{-1}\left(
\mathbf{z}_{{}}^{\ast},\mathbf{w}_{{}}^{\ast}\right) \\
&  =W^{-1}\left(  \mathbf{w}_{{}}^{\ast},\mathbf{z}_{{}}^{\ast}\right)
\left(  \mathbf{w}_{{}}^{\ast}\mathbf{w}_{{}}-I_{n}+I_{n}-\mathbf{w}_{{}%
}^{\ast}\mathbf{z}_{{}}+\mathbf{w}_{{}}^{\ast}\mathbf{z}_{{}}-\mathbf{w}_{{}%
}^{\ast}\mathbf{z}_{{}}\mathbf{z}_{{}}^{\ast}\mathbf{w}_{{}}\right)
W^{-1}\left(  \mathbf{z}_{{}}^{\ast},\mathbf{w}_{{}}^{\ast}\right) \\
&  =W^{-1}\left(  \mathbf{\mathbf{z}_{{}}^{\ast}},\mathbf{w}^{\ast}\right)
+W^{-1}\left(  \mathbf{w}_{{}}^{\ast},\mathbf{z}_{{}}^{\ast}\right)
\mathbf{w}_{{}}^{\ast}\mathbf{z}_{{}}-W^{-1}\left(  \mathbf{w}_{{}}^{\ast
},\mathbf{z}_{{}}^{\ast}\right)  \left(  I_{n}-\mathbf{w}_{{}}^{\ast
}\mathbf{w}_{{}}\right)  W^{-1}\left(  \mathbf{z}_{{}}^{\ast},\mathbf{w}_{{}%
}^{\ast}\right) \\
&  =W^{-1}\left(  \mathbf{z}_{{}}^{\ast},\mathbf{w}_{{}}^{\ast}\right)
+W^{-1}\left(  \mathbf{w}_{{}}^{\ast},\mathbf{z}_{{}}^{\ast}\right)
-I_{n}-W^{-1}\left(  \mathbf{w}_{{}}^{\ast},\mathbf{z}_{{}}^{\ast}\right)
\left(  I_{n}-\mathbf{w}_{{}}^{\ast}\mathbf{w}_{{}}\right)  W^{-1}\left(
\mathbf{z}_{{}}^{\ast},\mathbf{w}_{{}}^{\ast}\right)  .
\end{align*}
This gives $\left(  \ref{e2.25}\right)  $ and $\left(  \ref{e2.26}\right)  .$
Therefore, the proof of the lemma is complete.
\end{proof}

\begin{lemma}
\label{l2.7}The following statements hold.

\begin{enumerate}
\item[(i)] For $\mathbf{z}_{{}}\in\mathbf{II}\left(  n\right)  $ and
$\mathbf{w}_{{}}\in\overline{\mathbf{II}\left(  n\right)  }$, one has
\begin{equation}
D_{2}^{jk}:=\sum_{\alpha,\beta}\frac{V_{\alpha\beta}b_{j\alpha}\overline
{c_{k\beta}}}{\left(  1-\delta_{j\alpha}/2\right)  \left(  1-\delta_{k\beta
}/2\right)  }=4\left[  W^{-1}\left(  \mathbf{z}_{{}}^{\ast},\mathbf{w}_{{}%
}^{\ast}\right)  -I_{n}\right]  _{jk} \label{e2.28}%
\end{equation}
and%
\begin{equation}
E_{2}^{jk}:=\sum_{\alpha,\beta}\frac{V_{\alpha\beta}c_{j\alpha}\overline
{b_{k\beta}}}{\left(  1-\delta_{j\alpha}/2\right)  \left(  1-\delta_{k\beta
}/2\right)  }=4\left[  W^{-1}\left(  \mathbf{w}_{{}}^{\ast},\mathbf{z}_{{}%
}^{\ast}\right)  -I_{n}\right]  _{jk}. \label{e2.29}%
\end{equation}

\item[(ii)] For $\mathbf{\mathbf{z}_{{}}\in III}\left(  n\right)  $ and
$\mathbf{\mathbf{w}_{{}}\in\overline{\mathbf{III}\left(  n\right)  }}$, one
has
\begin{equation}
D_{3}^{jk}:=\sum_{\alpha,\beta}V_{\alpha\beta}\left(  1-\delta_{j\alpha
}\right)  \left(  1-\delta_{k\beta}\right)  b_{j\alpha}\overline{c_{k\beta}%
}=4\left[  W^{-1}\left(  \mathbf{z}_{{}}^{\ast},\mathbf{w}_{{}}^{\ast}\right)
-I_{n}\right]  _{jk}\;\; \label{e2.30}%
\end{equation}
and%
\begin{equation}
E_{3}^{jk}:=\sum_{\alpha,\beta}V_{\alpha\beta}\left(  1-\delta_{j\alpha
}\right)  \left(  1-\delta_{k\beta}\right)  c_{j\alpha}\overline{b_{k\beta}%
}=4\left[  W^{-1}\left(  \mathbf{w}_{{}}^{\ast},\mathbf{z}_{{}}^{\ast}\right)
-I_{n}\right]  _{jk}. \label{e2.31}%
\end{equation}

\end{enumerate}
\end{lemma}

\begin{proof}
\bigskip On $\mathbf{II}\left(  n\right)  $, by Part (i) of Lemma \ref{l2.4},
one has%
\begin{align*}
D_{2}^{jk}  &  =4\sum_{\alpha,\beta=1}^{n}V_{\alpha\beta}\left[
\mathbf{z}_{{}}^{\ast}V^{-1}\left(  \mathbf{z}_{{}}\right)  \right]
_{j\alpha}\left[  W^{-1}\left(  \mathbf{w}_{{}},\mathbf{z}_{{}}\right)
\mathbf{w}_{{}}\right]  _{k\beta}\\
&  =4\left[  \mathbf{z}_{{}}^{\ast}\left(  W^{-1}\left(  \mathbf{w}_{{}%
},\mathbf{z}_{{}}\right)  \mathbf{w}_{{}}\right)  ^{t}\right]  _{jk}\\
&  =4\left[  \mathbf{z}_{{}}^{\ast}\mathbf{w}_{{}}W^{-1}\left(  \mathbf{z}%
_{{}}^{\ast},\mathbf{w}_{{}}^{\ast}\right)  \right]  _{jk}\\
&  =4\left[  W^{-1}\left(  \mathbf{z}_{{}}^{\ast},\mathbf{w}_{{}}^{\ast
}\right)  -I_{n}\right]  _{jk}%
\end{align*}
and%
\begin{align*}
E_{2}^{jk}  &  =4\sum_{\alpha,\beta=1}^{n}V_{\alpha\beta}\left[
\mathbf{w}_{{}}^{\ast}W^{-1}\left(  \mathbf{z}_{{}},\mathbf{w}_{{}}\right)
\right]  _{j\alpha}\left[  V^{-1}\left(  \mathbf{z}_{{}}\right)
\mathbf{z}_{{}}\right]  _{k\beta}\\
&  =4\left[  \mathbf{w}^{\ast}W^{-1}\left(  \mathbf{z},\mathbf{w}\right)
V\left(  \mathbf{z}\right)  V^{-1}\left(  \mathbf{z}\right)  \mathbf{z}%
\right]  _{jk}\\
&  =4\left[  \mathbf{w}^{\ast}W^{-1}\left(  \mathbf{z},\mathbf{w}\right)
\mathbf{z}\right]  _{jk}\\
&  =4\left[  W^{-1}\left(  \mathbf{w}_{{}}^{\ast},\mathbf{z}_{{}}^{\ast
}\right)  -I_{n}\right]  _{jk},
\end{align*}
Part (i) is proved.

On $\mathbf{III}\left(  n\right)  $, $\mathbf{z}_{{}}^{\ast}V^{-1}\left(
\mathbf{z}_{{}}\right)  $ and $W^{-1}\left(  \mathbf{w}_{{}},\mathbf{z}_{{}%
}\right)  \mathbf{w}_{{}}$ are anti-symmetric, one has%
\begin{align*}
D_{3}^{jk}  &  =-4\sum_{\alpha,\beta=1}^{n}\left(  1-\delta_{j\alpha}\right)
\left(  1-\delta_{k\beta}\right)  V_{\alpha\beta}\left[  \mathbf{z}_{{}}%
^{\ast}V^{-1}\left(  \mathbf{z}_{{}}\right)  \right]  _{j\alpha}\left[
W^{-1}\left(  \mathbf{w}_{{}},\mathbf{z}_{{}}\right)  \mathbf{w}_{{}}\right]
_{k\beta}\\
&  =4\left[  \mathbf{z}_{{}}^{\ast}V^{-1}\left(  \mathbf{z}_{{}}\right)
V\left(  \mathbf{z}\right)  \mathbf{w}_{{}}W^{-1}\left(  \mathbf{z}_{{}}%
^{\ast},\mathbf{w}_{{}}^{\ast}\right)  \right]  _{jk}\\
&  =4\left[  \mathbf{z}_{{}}^{\ast}\mathbf{w}_{{}}W^{-1}\left(  \mathbf{z}%
_{{}}^{\ast},\mathbf{w}_{{}}^{\ast}\right)  \right]  _{jk}\\
&  =4\left[  W^{-1}\left(  \mathbf{z}_{{}}^{\ast},\mathbf{w}_{{}}^{\ast
}\right)  -I_{n}\right]  _{jk}%
\end{align*}
and%
\begin{align*}
E_{3}^{jk}  &  =-4\sum_{\alpha,\beta=1}^{n}\left(  1-\delta_{j\alpha}\right)
\left(  1-\delta_{k\beta}\right)  V_{\alpha\beta}\left[  \mathbf{w}_{{}}%
^{\ast}W^{-1}\left(  \mathbf{z}_{{}},\mathbf{w}_{{}}\right)  \right]
_{j\alpha}\left[  V^{-1}\left(  \mathbf{z}_{{}}\right)  \mathbf{z}_{{}%
}\right]  _{k\beta}\\
&  =4\left[  \mathbf{w}_{{}}^{\ast}W^{-1}\left(  \mathbf{z}_{{}}%
,\mathbf{w}_{{}}\right)  V\left(  \mathbf{z}_{{}}\right)  V^{-1}\left(
\mathbf{z}_{{}}\right)  \mathbf{z}_{{}}\right]  _{jk}\\
&  =4\left[  \mathbf{w}_{{}}^{\ast}W^{-1}\left(  \mathbf{z}_{{}}%
,\mathbf{w}_{{}}\right)  \mathbf{z}_{{}}\right]  _{jk}\\
&  =4\left[  \mathbf{w}_{{}}^{\ast}\mathbf{z}_{{}}W^{-1}\left(  \mathbf{w}%
_{{}}^{\ast},\mathbf{z}_{{}}^{\ast}\right)  \right]  _{jk}\\
&  =4\left[  W^{-1}\left(  \mathbf{w}_{{}}^{\ast},\mathbf{z}_{{}}^{\ast
}\right)  -I_{n}\right]  _{jk}.
\end{align*}
Therefore, the proof of the lemma is complete.
\end{proof}

\noindent\textbf{The proof of Theorem \ref{t2.2}}.

\begin{proof}
Part (i) of Theorem\textbf{ }\ref{t2.2} was proved by Hua \cite{Hua19630}. We
start to prove Part (ii).

On $\mathbf{II}\left(  n\right)  $, by Propositions \ref{p2.1} and \ref{p2.3},
Lemmas \ref{l2.5}, \ref{l2.6} and \ref{l2.7}, for $\mathbf{z}_{{}}%
\in\mathbf{II}\left(  n\right)  $ and $\mathbf{w}_{{}}\in\mathcal{U}\left(
\mathbf{II}\left(  n\right)  \right)  $, one has%
\begin{align*}
&  \Delta_{2}^{jk}P^{\mathbf{II}\left(  n\right)  }\left(  \mathbf{z}_{{}%
},\mathbf{w}_{{}}\right) \\
&  =A_{2}^{jk}+B_{2}^{jk}+C_{2}^{jk}+D_{2}^{jk}+E_{2}^{jk}\\
&  =-4V^{kj}+4\left[  V^{-1}\left(  \mathbf{z}_{{}}^{\ast}\right)
-I_{n}\right]  _{jk}+4\left[  W^{-1}\left(  \mathbf{z}_{{}}^{\ast}%
,\mathbf{w}_{{}}^{\ast}\right)  +W^{-1}\left(  \mathbf{w}_{{}}^{\ast
},\mathbf{z}_{{}}^{\ast}\right)  -I_{n}\right]  _{jk}\\
&  -4\left[  W^{-1}\left(  \mathbf{z}_{{}}^{\ast},\mathbf{w}_{{}}^{\ast
}\right)  -I_{n}\right]  _{jk}-4\left[  W^{-1}\left(  \mathbf{w}_{{}}^{\ast
},\mathbf{z}_{{}}^{\ast}\right)  -I_{n}\right]  _{jk}\\
&  =0.
\end{align*}

On $\mathbf{III}\left(  n\right)  $ and $n$ is even, by Proposition \ref{p2.1}
and \ref{p2.3}, Lemmas \ref{l2.5}, \ref{l2.6} and \ref{l2.7}, for
$\mathbf{z}_{{}}\in\mathbf{III}\left(  n\right)  $ and $\mathbf{w}_{{}}%
\in\mathcal{U}\left(  \mathbf{III}\left(  n\right)  \right)  $, one has%
\begin{align*}
&  \Delta_{3}^{jk}P^{\mathbf{III}\left(  n\right)  }\left(  \mathbf{z}_{{}%
},\mathbf{w}_{{}}\right) \\
&  =A_{3}^{jk}+B_{3}^{jk}+C_{3}^{jk}+D_{3}^{jk}+E_{3}^{jk}\\
&  =-4V^{kj}+4\left[  V^{-1}\left(  \mathbf{z}_{{}}^{\ast}\right)
-I_{n}\right]  _{jk}+4\left[  W^{-1}\left(  \mathbf{z}_{{}}^{\ast}%
,\mathbf{w}_{{}}^{\ast}\right)  +W^{-1}\left(  \mathbf{w}_{{}}^{\ast
},\mathbf{z}_{{}}^{\ast}\right)  -I_{n}\right]  _{jk}\\
&  -4\left[  W^{-1}\left(  \mathbf{z}_{{}}^{\ast},\mathbf{w}_{{}}^{\ast
}\right)  -I_{n}\right]  _{jk}-4\left[  W^{-1}\left(  \mathbf{w}_{{}}^{\ast
},\mathbf{z}_{{}}^{\ast}\right)  -I_{n}\right]  _{jk}\\
&  =0.
\end{align*}
Therefore, the proof of the theorem is complete by the Poisson integral
formula $\left(  \ref{e1.12}\right)  $ for $u.$
\end{proof}

\section{Proof of Theorem \ref{t1.3}}

In this section, we will prove Theorem \ref{t1.3} by using the idea based on
the argument in Graham \cite{graham19831}.

Denoted by $\tilde{\Delta}$ the modified Laplace-Beltrami operator in the unit
ball $B^{n}\subset\mathbb{C}^{n}:$%
\begin{equation}
\tilde{\Delta}:=\sum_{j,k=1}^{n}\left(  \delta_{jk}-\left\vert z\right\vert
^{2}z_{j}\bar{z}_{k}\right)  \frac{\partial^{2}}{\partial z_{j}\partial\bar
{z}_{k}}. \label{e3.1}%
\end{equation}
This is a new operator which is not included in the cases of the
Laplace-Beltrami operators studied in Graham and Lee \cite{graham88}.

\begin{theorem}
\label{t3.1}Let $n>1$, $p,q\in\mathbb{N\cup}\left\{  0\right\}  .$ Let
$f_{p,q}\left(  z\right)  =\sum_{\left\vert \alpha\right\vert =p,\left\vert
\beta\right\vert =q}a_{\alpha\bar{\beta}}z^{\alpha}\bar{z}^{\beta}$ be
harmonic in $B_{n}$ and $u\in C^{2}\left(  B_{n}\right)  $ such that
\[
\left\{
\begin{array}
[c]{cc}%
\tilde{\Delta}u=0, & \text{in }B_{n};\\
u=f_{p,q}, & \text{on }\partial B_{n}.
\end{array}
\right.
\]
Then

\begin{enumerate}
\item[(i)] If $n$ is odd and $u\in C^{\frac{n+1}{2}}\left(  \overline{B_{n}%
}\right)  $ then $pq=0;$

\item[(ii)] If $n$ is even and $u\in C^{\frac{n}{2},\alpha}\left(
\overline{B_{n}}\right)  $ for some $\alpha>1/2$ then $pq=0.$
\end{enumerate}
\end{theorem}

\begin{proof}
Following the argument of Graham \cite{graham19831}, we consider $h\left(
t\right)  $ on $\left[  0,1\right]  $ such that $h\left(  1\right)  =1$ and%
\[
\tilde{\Delta}\left(  h\left(  \left\vert z\right\vert ^{4}\right)
f_{p,q}\left(  z\right)  \right)  =0,\;\;z\in B_{n}.
\]
Notice that%

\begin{align*}
\frac{\partial^{2}\left(  f_{p,q}\left(  z\right)  h\left(  \left\vert
z\right\vert ^{4}\right)  \right)  }{\partial z_{j}\partial\bar{z}_{k}}  &
=f_{p,q}\left(  z\right)  \left[  2h^{\prime}\left(  \left\vert z\right\vert
^{4}\right)  \left(  z_{k}\bar{z}_{j}+\left\vert z\right\vert ^{2}\delta
_{jk}\right)  +4\left\vert z\right\vert ^{4}\bar{z}_{j}z_{k}h^{\prime\prime
}\left(  \left\vert z\right\vert ^{4}\right)  \right] \\
&  +h\left(  \left\vert z\right\vert ^{4}\right)  \frac{\partial^{2}f_{p,q}%
}{\partial z_{j}\partial\bar{z}_{k}}+2\left\vert z\right\vert ^{2}h^{\prime
}\left(  \left\vert z\right\vert ^{4}\right)  \left(  z_{k}\frac{\partial
f_{p,q}}{\partial z_{j}}+\bar{z}_{j}\frac{\partial f_{p,q}}{\partial z_{k}%
}\right)  ,
\end{align*}%
\begin{align*}
\lefteqn{\sum_{j,k=1}^{n}\left(  \delta_{jk}-\left\vert z\right\vert ^{2}%
z_{j}\bar{z}_{k}\right)  \left[  2h^{\prime}\left(  \left\vert z\right\vert
^{4}\right)  \left(  z_{k}\bar{z}_{j}+\left\vert z\right\vert ^{2}\delta
_{jk}\right)  +4\left\vert z\right\vert ^{4}\bar{z}_{j}z_{k}h^{\prime\prime
}\left(  \left\vert z\right\vert ^{4}\right)  \right]  }\\
&  =4h^{\prime\prime}\left(  \left\vert z\right\vert ^{4}\right)  \left\vert
z\right\vert ^{6}\left(  1-\left\vert z\right\vert ^{4}\right)  +2h^{\prime
}\left(  \left\vert z\right\vert ^{4}\right)  \left(  \left\vert z\right\vert
^{2}-\left\vert z\right\vert ^{6}+\left\vert z\right\vert ^{2}\left(
n-\left\vert z\right\vert ^{4}\right)  \right)
\end{align*}
and%
\begin{align*}
&  \sum_{j,k=1}^{n}\left(  \delta_{jk}-\left\vert z\right\vert ^{2}z_{j}%
\bar{z}_{k}\right)  \left(  2\left\vert z\right\vert ^{2}h^{\prime}\left(
\left\vert z\right\vert ^{4}\right)  \left(  z_{k}\frac{\partial f_{p,q}%
}{\partial z_{j}}+\bar{z}_{j}\frac{\partial f_{p,q}}{\partial z_{k}}\right)
+h\left(  \left\vert z\right\vert ^{4}\right)  \frac{\partial^{2}f_{p,q}%
}{\partial z_{j}\partial\bar{z}_{k}}\right) \\
&  =2\left\vert z\right\vert ^{2}h^{\prime}\left(  \left\vert z\right\vert
^{4}\right)  \left(  p+q\right)  \left(  1-\left\vert z\right\vert
^{4}\right)  f_{p,q}\left(  z\right)  -\left\vert z\right\vert ^{2}h\left(
\left\vert z\right\vert ^{4}\right)  pqf_{p,q}(z).
\end{align*}
Therefore,%
\begin{align*}
&  0=\tilde{\Delta}\left(  h\left(  \left\vert z\right\vert ^{4}\right)
\right)  f_{p,q}\left(  z\right) \\
&  =4h^{\prime\prime}\left(  \left\vert z\right\vert ^{4}\right)  \left\vert
z\right\vert ^{6}\left(  1-\left\vert z\right\vert ^{4}\right)  f_{p,q}\left(
z\right)  +2h^{\prime}\left(  \left\vert z\right\vert ^{4}\right)  \left(
\left\vert z\right\vert ^{2}-\left\vert z\right\vert ^{6}+\left\vert
z\right\vert ^{2}\left(  n-\left\vert z\right\vert ^{4}\right)  \right)
f_{p,q}\left(  z\right) \\
&  +2\left\vert z\right\vert ^{2}h^{\prime}\left(  \left\vert z\right\vert
^{4}\right)  \left(  p+q\right)  \left(  1-\left\vert z\right\vert
^{4}\right)  f_{p,q}-\left\vert z\right\vert ^{2}h\left(  \left\vert
z\right\vert ^{4}\right)  pqf_{p,q}\\
&  =4h^{\prime\prime}\left(  \left\vert z\right\vert ^{4}\right)  \left\vert
z\right\vert ^{6}\left(  1-\left\vert z\right\vert ^{4}\right)  f_{p,q}\left(
z\right) \\
&  +2h^{\prime}\left(  \left\vert z\right\vert ^{4}\right)  \left(  \left\vert
z\right\vert ^{2}\left(  n+1+\left(  p+q\right)  \right)  -\left\vert
z\right\vert ^{6}\left(  p+q+2\right)  \right)  f_{p,q}\left(  z\right) \\
&  -\left\vert z\right\vert ^{2}h\left(  \left\vert z\right\vert ^{4}\right)
pqf_{p,q}(z).
\end{align*}
With $t=|z|^{4}$, $h\left(  t\right)  $ satisfies the equation:%
\[
t\left(  1-t\right)  h^{\prime\prime}\left(  t\right)  +h^{\prime}\left(
t\right)  \left[  \frac{p}{2}+\frac{q}{2}+\frac{n+1}{2}-\left(  \frac{p}%
{2}+\frac{q}{2}+1\right)  t\right]  -\frac{p}{2}\frac{q}{2}h\left(  t\right)
=0.
\]
By the standard hypergeometric function theory \cite{graham19831} and
\cite{Slater66}, the smooth solution at $t=0$ must be
\begin{equation}
h\left(  t\right)  =\frac{F\left(  \frac{p}{2},\frac{q}{2},\frac{p+q+n+1}%
{2};t\right)  }{F\left(  \frac{p}{2},\frac{q}{2},\frac{p+q+n+1}{2};1\right)  }
\label{e3.4}%
\end{equation}
where%
\begin{equation}
F\left(  a,b,c;t\right)  =\sum_{n=0}^{+\infty}\frac{\left(  a\right)
_{n}\left(  b\right)  _{n}}{\left(  c\right)  _{n}}t^{n} \label{ee3.5}%
\end{equation}
and%
\begin{equation}
\left(  \alpha\right)  _{n}=\alpha\left(  \alpha+1\right)  \cdots\left(
\alpha+n-1\right)  . \label{ee3.6}%
\end{equation}

Assuming that $p,q>0,$ we will study the behavior of $h\left(  t\right)  $
near $t=1$ according to the value of $n.$

By the definition of $F\left(  a,b,c;t\right)  $ given by $\left(
\ref{ee3.5}\right)  $, it is easy to verify that%
\begin{equation}
\frac{d}{dt}F\left(  a,b,c;t\right)  =\frac{ab}{c}F\left(
a+1,b+1,c+1;t\right)  \;\text{if }abc\neq0. \label{ee3.7}%
\end{equation}

$\left(  \ref{ee3.7}\right)  $ and the following lemma about hypergeometric
function can be found in \cite{Slater66}.

\begin{lemma}
\label{l3.2}For $a,b,s>0$ with $a>s$ and $b>s,$ one has%
\begin{equation}
\lim_{t\rightarrow1^{-}}\frac{F\left(  a,b,a+b;t\right)  }{\log\frac{1}{1-t}%
}=\frac{\Gamma\left(  a+b\right)  }{\Gamma\left(  a\right)  \Gamma\left(
b\right)  }; \label{ee3.8}%
\end{equation}
Euler's identity:%
\begin{equation}
F\left(  a,b,a+b-s;t\right)  =\left(  1-t\right)  ^{-s}F\left(
b-s,a-s,a+b-s;t\right)  \label{ee3.9}%
\end{equation}
and%
\[
\lim_{t\rightarrow1^{-}}\left(  1-t\right)  ^{s}F\left(  a,b,a+b-s;t\right)
=\frac{\Gamma\left(  a+b-s\right)  \Gamma\left(  s\right)  }{\Gamma\left(
a\right)  \Gamma\left(  b\right)  }.
\]

\end{lemma}

To complete the proof of Theorem \ref{t3.1}, we need the following lemma.

\begin{lemma}
\label{l3.3}For $k\in\mathbb{N}$ and $p,q>0,$ one has

\begin{enumerate}
\item[(i)] There exist $H\in C^{k}\left(  \left[  0,1\right]  \right)  $ and
$G\in C^{k}\left(  \left[  0,1\right]  \right)  $ with $G\left(  1\right)
\neq0$ such that%
\begin{equation}
F\left(  \frac{p}{2},\frac{q}{2},\frac{2k+p+q}{2};t\right)  =H\left(
t\right)  +G\left(  t\right)  \left(  1-t\right)  ^{k}\log\left(  1-t\right)
. \label{ee3.11}%
\end{equation}

\item[(ii)] There exist $H\in C^{k+1}\left(  \left[  0,1\right]  \right)  $, a
constant $c\neq0$ such that%
\begin{equation}
F\left(  \frac{p}{2},\frac{q}{2},\frac{2k+p+q+1}{2};t\right)  =H\left(
t\right)  +c\left(  1-t\right)  ^{k+\frac{1}{2}}G_{1}\left(  t\right)  ,
\label{ee3.12}%
\end{equation}
where%
\begin{equation}
G_{1}\left(  t\right)  :=F\left(  k+\frac{q+1}{2},k+\frac{p+1}{2}%
,\frac{4k+p+q+3}{2};t\right)  . \label{ee3.13}%
\end{equation}
Moreover,%
\begin{equation}
\left(  1-t\right)  ^{k+\frac{1}{2}}G_{1}\left(  t\right)  \notin
C^{k+1}\left(  \left[  0,1\right]  \right)  . \label{ee3.14}%
\end{equation}

\end{enumerate}
\end{lemma}

\begin{proof}
Part (i) can be found in Graham \cite{graham19831}. The proof of Part (ii) may
be found through reading materials in \cite{Slater66}. For convenience for
readers, we sketch a proof here. By $\left(  \ref{ee3.7}\right)  $, one has%
\begin{align*}
F_{\ell}\left(  t\right)   &  :=\frac{d^{\ell}}{dt^{\ell}}F\left(  \frac{p}%
{2},\frac{q}{2},\frac{2k+p+q+1}{2};t\right) \\
&  =c_{\ell}F\left(  \ell+\frac{p}{2},\ell+\frac{q}{2},\ell+\frac{2k+p+q+1}%
{2};t\right)  ,
\end{align*}
where
\[
c_{0}=1\;\text{ and }c_{\ell}=\frac{\left(  \frac{p}{2}\right)  _{\ell}\left(
\frac{q}{2}\right)  _{\ell}}{\left(  \frac{2k+p+q+1}{2}\right)  _{\ell}}.
\]
Notice that $G_{1}\left(  t\right)  \in C\left(  \left[  0,1\right]  \right)
$. By $\left(  \ref{ee3.7}\right)  $ and $\left(  \ref{ee3.9}\right)  $, there
exists $\tilde{c}\neq0$ such that%
\[
\left(  1-t\right)  ^{\frac{1}{2}}\frac{d}{dt}G_{1}\left(  t\right)
=\tilde{c}F\left(  k+1+\frac{p}{2},k+1+\frac{q}{2},2k+\frac{p+q+5}{2}\right)
\in C\left(  \left[  0,1\right]  \right)  .
\]
Let%
\[
H_{1}\left(  t\right)  =c_{k}+\frac{c_{k+1}}{2}\int_{0}^{t}\left(  1-s\right)
^{\frac{1}{2}}\frac{dG_{1}\left(  s\right)  }{ds}ds.
\]
Then $H_{1}\left(  t\right)  \in C^{1}\left(  \left[  0,1\right]  \right)  .$
By $\left(  \ref{ee3.9}\right)  $ again,%
\[
F_{k+1}\left(  t\right)  =c_{k+1}\left(  1-t\right)  ^{-\frac{1}{2}}%
G_{1}\left(  t\right)  .
\]
The definition of $F_{\ell}$ implies%
\begin{align*}
F_{k}\left(  t\right)   &  =c_{k}+\int_{0}^{t}F_{k+1}\left(  s\right)
ds=c_{k}+c_{k+1}\int_{0}^{t}\left(  1-s\right)  ^{-\frac{1}{2}}G_{1}\left(
s\right)  ds\\
&  =H_{1}\left(  t\right)  -\frac{c_{k+1}}{2}\left(  1-t\right)  ^{\frac{1}%
{2}}G_{1}\left(  t\right)  .
\end{align*}
Let%
\[
H_{2}\left(  t\right)  =c_{k-1}+\int_{0}^{t}\left[  H_{1}\left(  s\right)
-\frac{1}{3}\left(  1-s\right)  ^{\frac{3}{2}}\frac{d}{ds}G_{1}\left(
s\right)  \right]  ds.
\]
Then $H_{2}\left(  t\right)  \in C^{2}\left(  \left[  0,1\right]  \right)  $
and%
\[
F_{k-1}=H_{2}\left(  t\right)  +\frac{c_{k+1}}{3}\left(  1-t\right)
^{\frac{3}{2}}G_{1}\left(  t\right)  .
\]
By induction, there exist $H\in C^{k+1}\left(  \left[  0,1\right]  \right)  $
and a constant $c\neq0$ such that%
\[
F\left(  \frac{p}{2},\frac{q}{2},\frac{2k+p+q+1}{2};t\right)  =H\left(
t\right)  +c\left(  1-t\right)  ^{k+\frac{1}{2}}G_{1}\left(  t\right)  .
\]
And%
\begin{align*}
\frac{d}{dt}\left(  \left(  1-t\right)  ^{\frac{1}{2}}G_{1}\left(  t\right)
\right)   &  =\frac{1}{2}\left(  1-t\right)  ^{-\frac{1}{2}}G_{1}\left(
t\right)  +\left(  1-t\right)  ^{\frac{1}{2}}\frac{dG_{1}}{dt}\left(  t\right)
\\
&  \notin C\left(  \left[  0,1\right]  \right)  .
\end{align*}
This implies that $\left(  1-t\right)  ^{k+\frac{1}{2}}G_{1}\left(  t\right)
\notin C^{k+1}\left(  \left[  0,1\right]  \right)  $ and the lemma is proved.
\end{proof}

Now we continue the proof of Theorem \ref{t3.1}. By $\left(  \ref{e3.4}%
\right)  ,$
\[
u\left(  z\right)  =h\left(  \left\vert z\right\vert ^{4}\right)
f_{p,q}\left(  z\right)  =\frac{F\left(  \frac{p}{2},\frac{q}{2}%
,\frac{p+q+n+1}{2};\left\vert z\right\vert ^{t}\right)  }{F\left(  \frac{p}%
{2},\frac{q}{2},\frac{p+q+n+1}{2};1\right)  }f_{p,q}\left(  z\right)  .
\]
We have the following two cases:

\begin{enumerate}
\item[(i)] When $n$ is odd and $pq\neq0$, by Part (i) of Lemma \ref{l3.2} with
$k=\frac{n+1}{2},$%
\[
h\left(  t\right)  =H\left(  t\right)  +G\left(  t\right)  \left(  1-t\right)
^{\frac{n+1}{2}}\log\left(  1-t\right)
\]
with $H,G\in C^{\infty}([0,1])$ and $G(1)\neq0$. Since $u\in C^{\frac{n+2}{2}%
}\left(  \overline{B_{n}}\right)  $ implies $h\left(  t\right)  \in
C^{\frac{n+1}{2}}\left(  \left[  0,1\right]  \right)  $. This is a
contradiction, which implies that $pq=0$. This proves the Part (i) of Theorem
\ref{t3.1}.

\item[(ii)] When $n=2k$ is even and $pq\neq0,$ by Part (ii) of Lemma
\ref{l3.2}, for any $\alpha>1/2,$ we know that $h\in C^{\infty}\left(
[0,1)\right)  \bigcap C^{k}\left(  \left[  0,1\right]  \right)  $ and%
\[
\left(  1-t\right)  ^{1-\alpha}\frac{d^{k+1}}{dt^{k+1}}h\left(  t\right)
\]
is unbounded on $\left[  0,1\right]  $. The assumption of Part (ii) of Theorem
\ref{t3.1} implies $h\in C^{k,\alpha}\left(  \left[  0,1\right]  \right)
\bigcap C^{\infty}\left(  [0,1)\right)  $ for some $\alpha>\frac{1}{2}$. This
is a contradiction, which implies that $pq=0$. This proves Part (ii) of
Theorem \ref{t3.1}.
\end{enumerate}

Therefore, the proof of Theorem \ref{t3.1} is complete.
\end{proof}

\noindent\textbf{Proof of Theorem \ref{t1.3}}

\begin{proof}
By the spherical harmonic expansions for $u$ on $\partial B_{n},$%
\[
u\left(  z\right)  =\sum_{p,q=0}^{+\infty}f_{p,q}\left(  z\right)
,\;\;z\in\partial B_{n}%
\]
where $f_{p,q}$ is a spherical harmonic function in $B_{n}$ of homogenous
degrees $\left(  p,q\right)  $. Then%
\[
u\left(  z\right)  =\sum_{p,q=0}^{+\infty}h_{p,q}\left(  \left\vert
z\right\vert ^{4}\right)  f_{p,q}\left(  z\right)  .
\]
By Theorem \ref{t3.1} and the assumption of Theorem \ref{t1.3}, one has that
$f_{p,q}=0$ if $pq\neq0$. This implies $u$ is pluriharmonic in $B_{n},$ and
the proof of Theorem \ref{t1.3} is complete.
\end{proof}

\section{Proof of Theorem \ref{t1.2}}

\noindent For a bounded domain $D\subset\mathbb{C}^{N}$, we use
$\operatorname{Aut}\left(  D\right)  $ to denote the automorphism group on
$D$. We say that $D$ is transitive or homogeneous if any two points $z,w\in D$
there is a $\phi\in\operatorname{Aut}(D)$ such that $\phi(z)=w$. $D$ is
symmetric if for any $z\in D$, there is $S_{z}\in\operatorname{Aut}(D)$ such
that $z$ is an isolated fixed point for $(S_{z})^{2}$.

\begin{proposition}
\label{p4.1}Let $D$ be a transitive domain in $\mathbb{C}^{N}$. Let
$\mathcal{A}(D)$ be a subset of $C^{2}(D)$ such that for any $\phi
\in\operatorname{Aut}(D)$ and $u\in\mathcal{A}(D)$ one has $u\circ\phi
\in\mathcal{A}(D)$. If there is a point $z_{0}\in D$ such that
\[
{\frac{\partial^{2}u(z_{0})}{\partial z_{j}\partial\overline{z}_{k}}}%
=0,\quad1\leq j,k\leq N,\ u\in\mathcal{A}(D),
\]
then $\mathcal{A}(D)$ is a subset of pluriharmonic functions on $D$.
\end{proposition}

\begin{proof}
Let $u\in\mathcal{A}(D)$ be an arbitrary element. Then for any $w\in D$, since
$D$ is transitive, there is a $\phi\in\operatorname{Aut}(D)$ such that
\[
\phi(z_{0})=w.
\]
Since $\mathcal{A}(D)$ is invariant under automorphism, one has that
$u\circ\phi\in\mathcal{A}(D)$ and
\[
{\frac{\partial^{2}u\circ\phi}{\partial z_{j}\partial\overline{z}_{k}}}%
(z_{0})=0,\quad1\leq j,k\leq N.
\]
Let $H_{u}$ be the complex Hessian matrix of $u$ and let $\phi^{\prime
}(z)=\left[  {\dfrac{\partial\phi_{k}}{\partial z_{j}}}\right]  $ be the
Jacobian matrix with index $j$ represents the row and $k$ represents the
column. Then
\begin{equation}
H_{u\circ\phi}(z_{0})=\phi^{\prime}(z_{0})H_{u}(\phi(z_{0}))\phi^{\prime
}(z_{0})^{\ast}. \label{e4.1}%
\end{equation}
Therefore
\[
H_{u}(w)=\phi^{\prime}(z_{0})^{-1}H_{u\circ\phi}(z_{0})(\phi^{\prime}%
(z_{0})^{\ast})^{-1}=0.
\]
This proves that $u$ is pluriharmonic in $D$.
\end{proof}

\begin{lemma}
\label{l4.2}If $A$ is an $n\times n$ matrix over $\mathbb{C}$ such that
\begin{equation}
\langle\xi,\xi A\rangle=0,\quad\text{for all }\xi\in\partial B_{n},
\label{e4.2}%
\end{equation}
then $A=0$.
\end{lemma}

\begin{proof}
\noindent Applying (\ref{e4.2}) to $\xi=\mathbf{e}_{k}$, one has $\left[
A\right]  _{kk}=0$ for all $1\leq k\leq n$. Then applying the identity
(\ref{e4.2}) to $\xi={\frac{1}{\sqrt{2}}}(\mathbf{e}_{k}+\mathbf{e}_{j})$ and
to $\xi={\frac{1}{\sqrt{2}}}(\mathbf{e}_{k}+\sqrt{-1}\mathbf{e}_{j})$,
respectively, one has
\[
\left[  A\right]  _{jk}+\left[  A\right]  _{kj}=0\quad\text{and }\left[
A\right]  _{jk}-\left[  A\right]  _{kj}=0\;\;\text{for }k\neq j.
\]
This implies $A=0.$
\end{proof}

\begin{theorem}
\label{t4.3}Let $m\leq n$ and $u\in C^{n}\left(  \overline{\mathbf{I}\left(
m,n\right)  }\right)  $ be invariant harmonic in $\mathbf{I}(m,n)$. Then
\begin{equation}
{\frac{\partial^{2}u}{\partial z_{j\alpha}\partial\overline{z}_{k\beta}}%
}(0)=0,\quad1\leq j,k\leq m,\ 1\leq\alpha,\beta\leq n. \label{e4.3}%
\end{equation}

\end{theorem}

\begin{proof}
\noindent For any $\lambda=(\lambda_{1},\cdots,\lambda_{n})\in B_{n}$ and
$\xi=(\xi_{1},\cdots,\xi_{m})\in\partial B_{m}$ is fixed, we let
\begin{equation}
\mathbf{z}_{{}}=\mathbf{z}_{{}}(\lambda):=\xi^{t}\lambda. \label{e4.4}%
\end{equation}
Then
\begin{equation}
\sum_{p=1}^{m}z\left(  \lambda\right)  _{pi}\overline{z}\left(  \lambda
\right)  _{pj}=\lambda_{i}\overline{\lambda}_{j}. \label{e4.5}%
\end{equation}
Let $g(\lambda)=u(\mathbf{z}_{{}}(\lambda))$. Then $g\in C^{n}(\overline
{B_{n}})$ and
\begin{equation}
{\frac{\partial^{2}g(\lambda)}{\partial\lambda_{i}\partial\overline{\lambda
}_{j}}}=\sum_{k,\ell=1}^{m}\sum_{\alpha,\beta=1}^{n}{\frac{\partial
^{2}u(\mathbf{z})}{\partial z_{k\alpha}\partial\overline{z}_{\ell\beta}}}%
\frac{\partial\left(  \xi_{k}\lambda_{\alpha}\right)  }{\partial\lambda_{i}%
}\frac{\partial\left(  \overline{\xi_{\ell}\lambda_{\beta}}\right)  }%
{\partial\overline{\lambda_{j}}}\label{e4.6}\\
=\sum_{k,\ell=1}^{m}{\frac{\partial^{2}u(\mathbf{z})}{\partial z_{ki}%
\partial\overline{z}_{\ell j}}}\xi_{k}\overline{\xi_{\ell}}.
\end{equation}
Therefore,
\begin{align*}
\lefteqn{\sum_{i,j=1}^{n}\left(  \delta_{ij}-\lambda_{i}\overline{\lambda_{j}%
}\right)  \frac{\partial^{2}g\left(  \lambda\right)  }{\partial\lambda
_{i}\partial\bar{\lambda}_{j}}}\\
&  =\sum_{k,\ell=1}^{m}\xi_{k}\overline{\xi_{\ell}}\sum_{i,j=1}^{n}%
(\delta_{ij}-\lambda_{i}\overline{\lambda_{j}}){\frac{\partial^{2}u}{\partial
z_{ki}\partial\overline{z}_{\ell j}}}\\
&  =\sum_{k,\ell=1}^{m}\xi_{k}\overline{\xi_{\ell}}\sum_{i,j=1}^{n}%
(\delta_{ij}-\sum_{p=1}^{m}z\left(  \lambda\right)  _{pi}\overline{z}\left(
\lambda\right)  _{pj}){\frac{\partial^{2}u}{\partial z_{ki}\partial
\overline{z}_{\ell j}}}\\
&  =\sum_{k,\ell=1}^{m}\xi_{k}\overline{\xi_{\ell}}\,(\Delta_{1}^{k\ell
}u)\circ(\mathbf{z}(\lambda))\\
&  =0.
\end{align*}
Graham Theorem for invariant harmonic function on $B_{n}$ implies that $g$ is
pluriharmonic in $B_{n}$. In particular, by (\ref{e4.6}), one has
\[
0={\frac{\partial^{2}g}{\partial\lambda_{i}\partial\overline{\lambda}_{j}}%
}(0)=\sum_{k,\ell=1}^{m}{\frac{\partial^{2}u(0)}{\partial z_{ki}%
\partial\overline{z}_{\ell j}}}\xi_{k}\overline{\xi_{\ell}},\quad1\leq
k,\ell\leq m,\ 1\leq i,j\leq n
\]
for all $\xi\in B_{n}$. Combining this and Lemma \ref{l4.2}, one has
\[
{\frac{\partial^{2}u}{\partial z_{ki}\partial\overline{z}_{\ell j}}%
}(0)=0,\quad1\leq k,\ell\leq m,1\leq i,j\leq n.
\]
The proof of Theorem \ref{t4.3} is complete.
\end{proof}

\begin{theorem}
\label{t4.4}If $n>1$ is either odd and $u\in C^{\frac{n+1}{2}}(\overline
{\mathbf{II}(n)})$ or $n=2k$ is even and $u\in C^{k,\alpha}(\overline
{\mathbf{II}(n)})$ for some $\alpha>1/2$ and if $u$ is invariant harmonic in
$\mathbf{II}(n)$, then
\begin{equation}
{\frac{\partial^{2}u}{\partial z_{j\alpha}\partial\overline{z}_{k\beta}}%
}(0)=0,\quad1\leq j,\alpha,k,\beta\leq n. \label{e4.7}%
\end{equation}

\end{theorem}

\begin{proof}
Let $\lambda=(\lambda_{1},\cdots,\lambda_{n})\in B_{n}$ and $U=[U_{jk}]$ is a
unitary matrix. Let
\begin{equation}
\mathbf{z}(\lambda)=(\lambda U)^{t}(\lambda U)\in\mathbf{II}(n) \label{e4.8}%
\end{equation}
and
\begin{equation}
v(\lambda)=u(\mathbf{z}(\lambda)). \label{e4.9}%
\end{equation}
Since $\mathbf{z}(\lambda)$ is holomorphic in $\lambda$ and $\mathbf{z}%
(\lambda)$ is symmetric, we have
\begin{align*}
\lefteqn{\frac{\partial^{2}v\left(  \lambda\right)  }{\partial\lambda_{\alpha
}\partial\bar{\lambda}_{\beta}}}\\
&  =\sum_{k,\ell=1}^{n}{\frac{1}{2-\delta_{k\ell}}}{\frac{\partial}%
{\partial\lambda_{\alpha}}}\left(  {\frac{\partial u}{\partial\overline
{z}_{k\ell}}}(\mathbf{z}(\lambda))\overline{\left[  \frac{\partial}%
{\partial\lambda_{\beta}}\left(  U^{t}\lambda^{t}\lambda U\right)  \right]
_{k\ell}}\right) \\
&  =\sum_{k,\ell=1}^{n}{\frac{1}{2-\delta_{k\ell}}}{\frac{\partial}%
{\partial\lambda_{\alpha}}}\left(  {\frac{\partial u}{\partial\overline
{z}_{k\ell}}}(\mathbf{z}(\lambda))\overline{\sum_{q}(\lambda_{q}U_{q\ell
}U_{\beta k}+U_{qk}\lambda_{q}U_{\beta\ell}})\right) \\
&  =\sum_{i,j=1}^{n}\sum_{k,\ell=1}^{n}{\frac{{\frac{\partial^{2}u}{\partial
z_{ij}\partial\overline{z}_{k\ell}}}(\mathbf{z}(\lambda))}{(2-\delta_{k\ell
})(2-\delta_{ij})}}\left(  \sum_{p}\lambda_{p}(U_{pj}U_{\alpha i}%
+U_{pi}U_{\alpha j})\right)  \left(  \overline{\sum_{q}\lambda_{q}(U_{q\ell
}U_{\beta k}+U_{qk}U_{\beta\ell}})\right) \\
&  =\sum_{i,j=1}^{n}\sum_{k,\ell=1}^{n}{\frac{1}{(1-{\frac{1}{2}}\delta
_{k\ell})(1-{\frac{1}{2}}\delta_{ij})}}{\frac{\partial^{2}u}{\partial
z_{ij}\partial\overline{z}_{k\ell}}}(\mathbf{z}(\lambda))(\sum_{p}\lambda
_{p}U_{pi}U_{\alpha j})\overline{\sum_{q}\lambda_{q}U_{qk}U_{\beta\ell}}\\
&  =\sum_{p,q}\lambda_{p}\overline{\lambda}_{q}\sum_{i,k=1}^{n}U_{pi}%
\overline{U_{qk}}\sum_{j,\ell=1}^{n}{\frac{1}{(1-{\frac{1}{2}}\delta_{k\ell
})(1-{\frac{1}{2}}\delta_{ij})}}{\frac{\partial^{2}u}{\partial z_{ij}%
\partial\overline{z}_{k\ell}}}(\mathbf{z}(\lambda))U_{\alpha j}\overline
{U_{\beta\ell}}%
\end{align*}
and
\begin{align*}
\sum_{\alpha,\beta=1}^{n}(\delta_{\alpha\beta}-|\lambda|^{2}\lambda_{\alpha
}\overline{\lambda}_{\beta})U_{\alpha j}\overline{U_{\beta\ell}}  &
=\delta_{j\ell}-|\lambda|^{2}\sum_{\alpha=1}^{n}\lambda_{\alpha}U_{\alpha
j}\sum_{\beta=1}^{n}\overline{\lambda_{\beta}U_{\beta\ell}}\\
&  =\delta_{j\ell}-\left[  \lambda U\right]  _{1j}\left\vert \lambda
U\right\vert ^{2}\left[  \overline{\lambda U}\right]  _{1\ell}\\
&  =\delta_{j\ell}-\left[  \left(  \lambda U\right)  ^{t}\right]  _{j1}\lambda
U\cdot\left(  \lambda U\right)  ^{\ast}\left[  \overline{\lambda U}\right]
_{1\ell}\\
&  =\delta_{j\ell}-\sum_{\alpha=1}^{n}z(\lambda)_{pj}\overline{z(\lambda
)_{p\ell}}\\
&  =V_{j\ell}(\mathbf{z}(\lambda)).
\end{align*}
Therefore,
\begin{align*}
\lefteqn{\sum_{\alpha,\beta=1}^{n}\left(  \delta_{\alpha\beta}-\left\vert
\lambda\right\vert ^{2}\lambda_{\alpha}\bar{\lambda}_{\beta}\right)
\frac{\partial^{2}v\left(  \lambda\right)  }{\partial\lambda_{\alpha}%
\partial\bar{\lambda}_{\beta}}}\\
&  =\sum_{p,q}\lambda_{p}\overline{\lambda}_{q}\sum_{i,k=1}^{n}U_{pi}%
\overline{U_{qk}}\sum_{j,\ell=1}^{n}{\frac{V_{j\ell}(\mathbf{z}(\lambda
))}{(1-{\frac{1}{2}}\delta_{k\ell})(1-{\frac{1}{2}}\delta_{ij})}}%
{\frac{\partial^{2}u}{\partial z_{ij}\partial\overline{z}_{k\ell}}}%
(\mathbf{z}(\lambda))\\
&  =\sum_{p,q}\lambda_{p}\overline{\lambda}_{q}\sum_{i,k=1}^{n}U_{pi}%
\overline{U_{qk}}\cdot\Delta_{2}^{ik}u(\mathbf{z}(\lambda))\\
&  =0.
\end{align*}
Applying Theorem \ref{t1.3} to $v$ on $B_{n}$, one has $v$ is pluriharmonic in
$B_{n}$. Thus,
\begin{align*}
0  &  ={\frac{\partial v(\lambda)}{\partial\lambda_{\alpha}\partial
\overline{\lambda}_{\beta}}}\\
&  =\sum_{p,q}\lambda_{p}\overline{\lambda}_{q}\sum_{i,k=1}^{n}U_{pi}%
\overline{U_{qk}}\sum_{j,\ell=1}^{n}{\frac{1}{(1-{\frac{1}{2}}\delta_{k\ell
})(1-{\frac{1}{2}}\delta_{ij})}}{\frac{\partial^{2}u}{\partial z_{ij}%
\partial\overline{z}_{k\ell}}}(\mathbf{z}(\lambda))U_{\alpha j}\overline
{U_{\beta\ell}}.
\end{align*}
For any $\xi\in\partial B_{n}$ and $\omega\in\left(  0,1\right)  $, by letting
$\lambda=\omega\xi U^{\ast}$ one has%
\[
\omega^{2}\sum_{i,k=1}^{n}\xi_{i}\bar{\xi}_{k}\sum_{j,\ell=1}^{n}{\frac
{1}{(1-{\frac{1}{2}}\delta_{k\ell})(1-{\frac{1}{2}}\delta_{ij})}}%
{\frac{\partial^{2}u}{\partial z_{ij}\partial\overline{z}_{k\ell}}}(\omega
^{2}\xi^{t}\xi)U_{\alpha j}\overline{U_{\beta\ell}}=0.
\]
Let $\omega\rightarrow0,$ one obtain%
\[
\sum_{i,k=1}^{n}\xi_{i}\bar{\xi}_{k}\sum_{j,\ell=1}^{n}{\frac{1}{(1-{\frac
{1}{2}}\delta_{k\ell})(1-{\frac{1}{2}}\delta_{ij})}}{\frac{\partial^{2}%
u}{\partial z_{ij}\partial\overline{z}_{k\ell}}}(0)U_{\alpha j}\overline
{U_{\beta\ell}}=0.
\]
Let
\[
A=\left[  A_{ki}\right]  :=\left[  \sum_{j,\ell=1}^{n}{\frac{1}{(1-{\frac
{1}{2}}\delta_{k\ell})(1-{\frac{1}{2}}\delta_{ij})}}{\frac{\partial^{2}%
u}{\partial z_{ij}\partial\overline{z}_{k\ell}}}(0)U_{\alpha j}\overline
{U_{\beta\ell}}\right]  .
\]
Then
\[
\langle\xi A,\xi\rangle=0.
\]
By Lemma \ref{l4.2}, this implies $A=0$. Therefore,
\[
\sum_{j,\ell=1}^{n}{\frac{1}{(1-{\frac{1}{2}}\delta_{k\ell})(1-{\frac{1}{2}%
}\delta_{ij})}}{\frac{\partial^{2}u}{\partial z_{ij}\partial\overline
{z}_{k\ell}}}(\mathbf{0}))U_{\alpha j}\overline{U_{\beta\ell}}=0.
\]
Take $\alpha=\beta$, Lemma \ref{l4.2} implies that
\[
{\frac{1}{(1-{\frac{1}{2}}\delta_{k\ell})(1-{\frac{1}{2}}\delta_{ij})}}%
{\frac{\partial^{2}u}{\partial z_{ij}\partial\overline{z}_{k\ell}}}%
(0)=0,\quad1\leq i,j,k,\ell\leq n.
\]
Therefore,
\[
{\frac{\partial^{2}u}{\partial z_{ij}\partial\overline{z}_{k\ell}}}%
(0)=0,\quad1\leq i,j,k,\ell\leq n,
\]
and the proof of the theorem is complete.
\end{proof}

\medskip

\begin{theorem}
\label{t4.5}Let $n>1$ be even and $u\in C^{n-1}\left(  \overline
{\mathbf{III}\left(  n\right)  }\right)  $ is invariant harmonic in
$\mathbf{III}(n)$. Then
\begin{equation}
{\frac{\partial^{2}u(0)}{\partial z_{j\alpha}\partial\overline{z}_{k\beta}}%
}=0,\quad1\leq j,k,\alpha,\beta\leq n. \label{e4.10}%
\end{equation}

\end{theorem}

\begin{proof}
Let $\mathbf{z}(\lambda):B_{n-1}\rightarrow\mathbf{III}(n)$ be defined by
\begin{equation}
\mathbf{z}(\lambda)=\left[
\begin{array}
[c]{cc}%
0 & \lambda\\
-\lambda^{t} & O_{n-1}%
\end{array}
\right]  , \label{e4.11}%
\end{equation}
where $O_{n-1}$ is an $(n-1)\times(n-1)$ zero matrix. Let
\[
g(\lambda)=u(\mathbf{z}(\lambda)),\quad\lambda=(\lambda_{2},\cdots,\lambda
_{n})\in B_{n-1}.
\]
Then
\begin{align*}
{\frac{\partial g(\lambda)}{\partial\lambda_{p}}}  &  =\sum_{j<\alpha}%
^{n}{\frac{\partial u(\mathbf{z}(\lambda))}{\partial z_{j\alpha}}}%
{\frac{\partial(z_{j\alpha}(\lambda)-z_{\alpha j}(\lambda))}{\partial
\lambda_{p}}}\\
&  =\sum_{j,\alpha=1}^{n}{\frac{\partial u(\mathbf{z}(\lambda))}{\partial
z_{j\alpha}}}{\frac{\partial z_{j\alpha}(\lambda)}{\partial\lambda_{p}}}\\
&  =2{\frac{\partial u(\mathbf{z}(\lambda)}{\partial z_{1p}}}%
\end{align*}
and
\[
{\frac{\partial^{2}g(\lambda)}{\partial\lambda_{p}\partial\overline{\lambda
}_{q}}}=4{\frac{\partial^{2}u(\mathbf{z}(\lambda)}{\partial z_{1p}%
\partial\overline{z}_{1q}}}.
\]

Since $V(\mathbf{z}(\lambda))=I_{n}-\mathbf{z}(\lambda)\mathbf{z}%
(\lambda)^{\ast}$, one has
\[
V_{p1}=(1-|\lambda|^{2})\delta_{p1},V_{1p}=(1-|\lambda|^{2})\delta_{1p}%
\]
and%
\[
\quad V_{\alpha\beta}=\delta_{\alpha\beta}-\lambda_{\alpha}\overline{\lambda
}_{\beta}\text{ for }\alpha,\beta\geq2.
\]
Therefore,
\begin{align*}
\lefteqn{\sum_{p,q=2}^{n}\left(  \delta_{pq}-\lambda_{p}\bar{\lambda}%
_{q}\right)  \frac{\partial^{2}g\left(  \lambda\right)  }{\partial\lambda
_{p}\partial\bar{\lambda}_{q}}}\\
&  =4\sum_{p,q=2}^{n}(\delta_{pq}-\lambda_{p}\overline{\lambda}_{q}%
){\frac{\partial^{2}u(\lambda)}{\partial z_{1p}\partial\overline{z}_{1q}}}\\
&  =4\Delta_{3}^{11}u(\mathbf{z}(\lambda))=0.
\end{align*}
By the Graham's theorem on $B_{n-1}$, one have $g$ is pluriharmonic in
$\lambda\in B_{n-1}$. For any unitary matrix $U$ since
\[
u_{U}(\mathbf{z})=u(U^{t}\mathbf{z}U)
\]
is also invariant harmonic in $\mathbf{III}(n)$. By the argument, if we let
\[
g(\lambda)=u(U^{t}\mathbf{z}(\lambda)U)
\]
then
\[
{\frac{\partial^{2}g(0)}{\partial\lambda_{p}\partial\bar{\lambda}_{q}}}=0.
\]
Notice that
\begin{align*}
{\frac{\partial g(\lambda)}{\partial\lambda_{p}}} &  =\sum_{j<\alpha}%
^{n}{\frac{\partial u(\mathbf{z}(\lambda)}{\partial z_{j\alpha}}}%
{\frac{\partial(z_{j\alpha}(\lambda)-z_{\alpha j}(\lambda))}{\partial
\lambda_{p}}}\\
&  =\sum_{j,\alpha=1}^{n}{\frac{\partial u(\mathbf{z}(\lambda)}{\partial
z_{j\alpha}}}{\frac{\partial z_{j\alpha}(\lambda)}{\partial\lambda_{p}}}\\
&  =\sum_{j,\alpha=1}^{n}{\frac{\partial u(\mathbf{z}(\lambda)}{\partial
z_{j\alpha}}}(U_{1j}U_{p\alpha}-U_{pj}U_{1\alpha}),
\end{align*}%
\begin{align*}
{\frac{\partial^{2}g(\lambda)}{\partial\lambda_{p}\partial\overline{\lambda
}_{q}}} &  =\sum_{j,\alpha=1}^{n}\sum_{k,\beta=1}^{n}{\frac{\partial
^{2}u(\mathbf{z}(\lambda)}{\partial z_{j\alpha}\partial\overline{z}_{k\beta}}%
}{\frac{\partial z_{j\alpha}(\lambda)}{\partial\lambda_{p}}}\overline
{\frac{\partial z_{k\beta}}{\partial\lambda_{q}}}\\
&  =\sum_{j,\alpha=1}^{n}\sum_{k,\beta=1}^{n}{\frac{\partial^{2}%
u(\mathbf{z}(\lambda)}{\partial z_{j\alpha}\partial\overline{z}_{k\beta}}%
}(U_{1j}U_{p\alpha}-U_{pj}U_{1\alpha})(\overline{U}_{1k}\overline{U}_{q\beta
}-\overline{U}_{qk}\overline{U}_{1\beta})
\end{align*}
and since $\mathbf{z}$ is anti-symmetric, one has
\begin{align*}
0 &  =\sum_{p,q=1}^{n}U_{q\ell}\overline{U}_{pm}{\frac{\partial^{2}%
g(0)}{\partial\lambda_{p}\partial\overline{\lambda}_{q}}}\\
&  =\sum_{j,\alpha=1}^{n}\sum_{k,\beta=1}^{n}{\frac{\partial^{2}u(0)}{\partial
z_{j\alpha}\partial\overline{z}_{k\beta}}}(U_{1j}\delta_{m\alpha}-\delta
_{mj}U_{1\alpha})(\overline{U}_{1k}\delta_{\ell\beta}-\delta_{\ell k}%
\overline{U}_{1\beta})\\
&  =\sum_{j,k=1}^{n}{\frac{\partial^{2}u(0)}{\partial z_{jm}\partial
\overline{z}_{k\ell}}}U_{1j}\overline{U}_{1k}-\sum_{j,\beta=1}^{n}%
{\frac{\partial^{2}u(0)}{\partial z_{jm}\partial\overline{z}_{\ell\beta}}%
}U_{1j}\overline{U}_{1\beta}\\
&  -\sum_{\alpha,k=1}^{n}{\frac{\partial^{2}u(0)}{\partial z_{m\alpha}%
\partial\overline{z}_{k\ell}}}U_{1\alpha}\overline{U}_{1k}+\sum_{\alpha,\beta
}{\frac{\partial^{2}u(0)}{\partial z_{m\alpha}\partial\overline{z}_{\ell\beta
}}}U_{1\alpha}\overline{U}_{1\beta}\\
&  =4\sum_{j,k=1}^{n}{\frac{\partial^{2}u(0)}{\partial z_{jm}\partial
\overline{z}_{k\ell}}}U_{1j}\overline{U}_{1k}.
\end{align*}
Lemma \ref{l4.2} implies
\[
{\frac{\partial^{2}u(0)}{\partial z_{jm}\partial\overline{z}_{k\ell}}}=0.
\]
The proof of the theorem is complete.
\end{proof}

\medskip

\noindent\textbf{Proof of Theorem 1.2} \medskip

\begin{proof}

\begin{enumerate}
\item[1)] If $u\in C^{n}(\overline{\mathbf{I}(m,n)})$ is invariant harmonic in
$\mathbf{I}(m,n)$. By Theorem \ref{t4.3} and Proposition \ref{p4.1}, one has
\[
\frac{\partial^{2}u}{\partial z_{j\alpha}\partial\bar{z}_{k\beta}}\left(
\mathbf{z}\right)  =0,\quad\mathbf{z}\in\mathbf{I}\left(  m,n\right)
,\quad1\leq j,k\leq m,1\leq\alpha,\beta\leq n.
\]
This means that $u$ is pluriharmonic in $\mathbf{I}(m,n)$.

\item[2)] If $n>1$ is odd and if $u\in C^{\frac{n+1}{2}}(\overline
{\mathbf{II}(n)})$ or $n=2k>1$ is even and if $u\in C^{k,\alpha}%
(\overline{\mathbf{II}(n)})$ for some $\alpha>1/2$ and if $u$ is invariant
harmonic in $\mathbf{II}(n)$. By Theorem \ref{t4.4} and Proposition
\ref{p4.1}, one has
\[
\frac{\partial^{2}u}{\partial z_{j\alpha}\partial\bar{z}_{k\beta}}\left(
\mathbf{z}\right)  =0,\quad\mathbf{z}\in\mathbf{II}\left(  n\right)
,\quad1\leq j,k,\alpha,\beta\leq n.
\]
This means that $u$ is pluriharmonic in $\mathbf{II}(n)$.

\item[3)] If $n$ is even and if $u\in C^{n-1}(\overline{\mathbf{III}(n)})$ is
invariant harmonic in $\mathbf{III}(n)$. By Theorem \ref{t4.5} and Proposition
\ref{p4.1}, one has
\[
\frac{\partial^{2}u}{\partial z_{j\alpha}\partial\bar{z}_{k\beta}}\left(
\mathbf{z}\right)  =0,\quad\mathbf{z}\in\mathbf{III}\left(  n\right)
,\quad1\leq j,k,\alpha,\beta\leq n.
\]
This means that $u$ is pluriharmonic in $\mathbf{III}(n)$.
\end{enumerate}

Therefore, the proof of Theorem \ref{t1.2} is complete.
\end{proof}

\section{Remarks on $\mathbf{III}(3),\mathbf{IV}(4)$ and $\mathbf{IV}(2)$}

\noindent\textbf{First}, let us make a remark on $\mathbf{III}(2k+1)$ when
$k=1$. It is known from Lu\thinspace\cite{lu1963} that $\mathbf{III}(3)$ is
biholomorphical to $B_{3}$. In fact, if let
\[
\varphi(z)=\left[
\begin{array}
[c]{ccc}%
0 & z_{1} & z_{2}\\
-z_{1} & 0 & z_{3}\\
-z_{2} & -z_{3} & 0
\end{array}
\right]
\]
then it is easy to verify that $\phi:B_{3}\rightarrow\mathbf{III}(3)$ is a
biholomorphic map.

\medskip

By Theorem \ref{t1.2} or the Graham's theorem on $B_{3}$, one has

\begin{corollary}
If $u\in C^{3}(\overline{\mathbf{III}(3)})$ is invariant harmonic then $u$ is
pluriharmonic in $\mathbf{III}(3)$.
\end{corollary}

\noindent\textbf{Second}, it is known from Lu\thinspace\cite{lu1963} that
$\mathbf{IV}(4)$ is biholomorphic to $\mathbf{I}(2,2)$. By Theorem \ref{t1.2},
one has the following corollary.

\begin{corollary}
If $u\in C^{2}(\overline{\mathbf{IV}(4)})$ is invariant harmonic in
$\mathbf{IV}(4)$ then $u$ is pluriharmonic in $\mathbf{IV}(4)$.
\end{corollary}

\noindent\textbf{Finally}, it is known from Lu\thinspace\cite{lu1963} that
$\mathbf{IV}(2)$ is biholomorphic to the polydisc $D(0,1)^{2}$ in
$\mathbb{C}^{2}$. Moreover, one can verify that the following map
\begin{equation}
\left(  w_{1},w_{2}\right)  =\varphi\left(  z\right)  =\left(  z_{1}%
+iz_{2},z_{1}-iz_{2}\right)  :D(0,1)^{2}\rightarrow\mathbf{IV}(2) \label{l5.1}%
\end{equation}
is a biholomorphic map. Applying the result in Li-Simon \cite{li02}, one has
the following result.

\begin{corollary}
If $u\in C\left(  \overline{\mathbf{IV}\left(  2\right)  }\right)  $ is
invariant harmonic in $\mathbf{IV}(2)$ then

\begin{enumerate}
\item[(i)] $u$ is harmonic (in the regular sense) in $\mathbf{IV}(2)$;

\item[(ii)] $2\operatorname{Re}\frac{\partial^{2}u}{\partial w_{1}\partial
\bar{w}_{2}}=0$ in $\mathbf{IV}(2)$;

\item[(iii)] $u$ is not pluriharmonic in general.
\end{enumerate}
\end{corollary}

\begin{proof}
Since
\begin{equation}
z_{1}=\frac{w_{1}+w_{2}}{2},z_{2}=\frac{w_{1}-w_{2}}{2i} \label{e5.2}%
\end{equation}
and let
\begin{equation}
v(z)=u(w). \label{e5.3}%
\end{equation}
Then $v$ is invariant harmonic in $D(0,1)^{2}$ and continuous up to the
boundary. By the result in \cite{li02}, we have
\begin{equation}
\frac{\partial^{2}v}{\partial z_{1}\partial\bar{z}_{1}}=\frac{\partial^{2}%
v}{\partial z_{2}\partial\bar{z}_{2}}=0,\quad z\in D(0,1)^{2}. \label{e5.4}%
\end{equation}
Notice that%
\begin{equation}
4\frac{\partial^{2}v}{\partial z_{1}\partial\bar{z}_{1}}=\frac{\partial^{2}%
u}{\partial w_{1}\partial\bar{w}_{1}}+\frac{\partial^{2}u}{\partial
w_{1}\partial\bar{w}_{2}}+\frac{\partial^{2}u}{\partial w_{2}\partial\bar
{w}_{1}}+\frac{\partial^{2}u}{\partial w_{2}\partial\bar{w}_{2}} \label{e5.5}%
\end{equation}
and%
\begin{equation}
-4\frac{\partial^{2}v}{\partial z_{2}\partial\bar{z}_{2}}=\frac{\partial^{2}%
u}{\partial w_{1}\partial\bar{w}_{1}}-\frac{\partial^{2}u}{\partial
w_{1}\partial\bar{w}_{2}}-\frac{\partial^{2}u}{\partial w_{2}\partial\bar
{w}_{1}}+\frac{\partial^{2}u}{\partial w_{2}\partial\bar{w}_{2}}. \label{e5.6}%
\end{equation}
By $\left(  \ref{e5.3}\right)  -\left(  \ref{e5.6}\right)  $, one can easily
see Parts (i) and (ii) of the above corollary holds.

In order to prove Part (iii), we let
\begin{equation}
u(z_{1},z_{2})=|z_{1}|^{2}-|z_{2}|^{2}. \label{e5.7}%
\end{equation}
One can verify that $u$ is invariant harmonic in $\mathbf{IV}(2)$, but it is
clearly that $u$ is not pluriharmonic in $\mathbf{IV}(2)$. Therefore, the
proof of the corollary is complete.
\end{proof}

\bibliographystyle{acm}
\bibliography{chen-li}

\end{document}